\documentclass{imsart}
\RequirePackage{natbib}
\usepackage{amsmath,amssymb,amsthm}
\usepackage{graphics,epsfig}
\usepackage{hyperref}
\usepackage{natbib}
\usepackage{color}
\usepackage{graphicx}
\usepackage{caption}
\usepackage{subcaption}
\usepackage{float}
\usepackage{dsfont}
\usepackage{mathrsfs}
\usepackage{multirow}
\usepackage{comment}
\usepackage{bm}

\def \ra {\rightarrow}

\def \E {\mathbb{E}}

\def \a {\alpha}
\def \be {\beta}

\newtheorem{definition}{\bf Definition}
\newtheorem{defn}[definition]{\bf Definition}

	\newtheorem{theorem}{\bf Theorem}
	
	\newtheorem{prop}{\bf Proposition}
	\newtheorem{lem}[theorem]{\bf Lemma}
	\newtheorem{as}{\bf Assumption}

	\newtheorem{coro}[theorem]{\bf Corollary}
	
\renewcommand{\epsilon}{\varepsilon}

\begin{document}

\begin{frontmatter}

\title{Motivation to Run in One-Day Cricket}
\runtitle{One-Day Cricket}

\begin{aug}

\author{\fnms{Paramahansa} 
	\snm{Pramanik}
	\ead[label=e1]{ppramanik1@niu.edu}}
\and
\author{\fnms{Alan M.} 
	\snm{Polansky}
	\ead[label=e2]{polansky@niu.edu}}
	
\runauthor{P. Pramanik and A. M. Polansky}

\affiliation{Northern Illinois University}

\address{Department of Statistics and Actuarial Science \\
	De Kalb, IL 60115 USA}
\end{aug}
  
\begin{abstract}
In this paper we introduce  a new methodology to determine an optimal coefficient for a positive finite measure of batting average, strike rate, and bowling average of a player in order to get an optimal score of a team under dynamic modeling using a path integral method. We also introduce new run dynamics modeled as a stochastic differential equation in order to incorporate the average weather conditions at the cricket ground, the weather condition on the day of the match including sudden deterioration which leads to a partial or complete stop of the game, total attendance, and home field advantage.
\end{abstract}



\end{frontmatter}

\section{Introduction}
Recent years have seen increased interest in establishing mathematical models for cricket. The introduction of T-$20$ cricket has increased the popularity of the sport in the Asian subcontinent, Australia, England, New Zealand, and West Indies. Therefore, a more sophisticated model is required to predict a team's score even when a match is stopped because of rain. Methods like the average run rate, most productive overs, discounted most productive overs, parabola, Clark curves, Duchworth/Lewis method, and the modified Duchworth/Lewis method have been used to set up a target score if a match is stopped because of rain \citep{duckworth1998}. Dynamic programming methods have been used by \cite{clarke1988}, and \cite{johnston1993}.

In this paper we consider modeling the discounted score of a player for last $10$ matches using hyperbolic discounting,  giving the highest weight to the last match and least weight corresponding to $10^{th}$ last match before the current game. This is motivated by the reasoning that a player's expected performance on the current game will depend more on his recent past performance instead of his far past performance. Then we define an objective function based on the performance of $11$ players in a team subject to stochastic differential run dynamics with finite drift and diffusion components. Due to Hebbian learning humans can be considered automatons in the sense that when a person sees certain objects, their outer neurons expand or contract in a certain unique way and send signals to the inner neurons through a synaptic system \citep{kappen2007}. As outer neurons send electrons to inner neurons, they can have infinite possible paths in the synaptic system and hence mimics \cite{feynman1949} path integral method, and by the Riemann-Lebesgue Lemma makes the path integral measurable \citep{de2012}. Once the signals come to the inner neurons, the person can observe the object, and is able to make decisions about it. Therefore, decisions of a player is the realization  of a stochastic process. Following \cite{chow1996} and \cite{yeung2006} we know Feynman path integral method is it provides an easier solution instead of going through the difficult Hamiltonian-Jacobi-Bellman equation .

\section{Methodology}
There are $11$ players in each team and let $\mathbf Z$ denote the run vector of all the players in each team. Player $i$'s score is given by
 \[
 u_i(\mathbf{Z})=\sum_{m=1}^M\exp(-\rho_i m)Z_{im}(u,w),
 \]
 where player $i$'s discount factor is $\rho_i\in(0,1]$, state variable $Z_{im}(u,w)\geq0$ is player $i$'s score before match $M+1$ (the current match) is a function of over remaining $u$ and total wickets lost $w\in[0,9]$, $M$ is the total number of matches played before match $M+1$. From the previous section we assume $M=10$. Another important assumption of this model is that, the discount factor $\rho_i$ is constant only for player $i$. That is for $i\neq j$ we assume that $\rho_i\neq\rho_j$.
The expected run of team $T$ before match $M+1$ starts is 
\begin{equation*}\label{0}
\mathbf{Z}_T=\E_{M} \sum_{i=1}^{I}\ \sum_{m=1}^M\beta_i W_i(u) \exp(-\rho_i m) Z_{im}(u,w),
\end{equation*}
where the control variable $W_i(u)\in\mathbb{R}^+$ is a measure of the valuation of the $i^{th}$ player in terms of reputation (such as higher batting average, strike rate for a batsman, and lower bowling average for a bowler), $I=11$, $\beta_i$ is the coefficient of $W_i(u)$ and $\E_M$ is the overall conditional expectation on $Z$ until match $M$. If player $i$ has a batting average more than $50$ and the strike rate is greater than $85$, then the measure $W_i(u)$ takes on a very large finite value. If Team $T$ loses couple of early wickets with very low score, it sends a batsman who can stay on the wicket longer with low strike rate (i.e. $W_i(u)\ra 0$) rather than a hard hitter or a pinch hitter. Finally, we assume $\mathbf{Z}_T\geq 0$ is a $C^\infty$ function with respect to $Z_{im}$ and $W_i$.

Assume the run dynamics of a match follow a stochastic differential equation
\begin{align}\label{1}
d\mathbf{Z}(u,w)=\hat{\bm{\mu}}[u,\mathbf{W}(u),\mathbf{Z}(u,w)]du+
\bm{\sigma}[u,\bm{\sigma}_2^*,\mathbf{W}(u),\mathbf{Z}(u,w)]d\mathbf{B}(u),
\end{align}
where $\mathbf{W}_{I\times 1}(u)=[W_{1}(u)\ W_{2}(u)\ ...\ W_{I}(u)]^T\subset \mathbb{W}\subset\mathbb{R}^I$ is the player control space, $\mathbf{Z}_{I\times 1}(u,w)=[Z_{1m}(u,w)\ Z_{2m}(u,w)\ ...\ Z_{Im}(u,w)]^T\subset\mathbb{Z}\subset\mathbb{R}^{I}$ is the run space under the one-day cricket rules, $\mathbf{B}_{p\times 1}(u)$ is a $p$-dimensional Brownian motion, $\bm{\mu}_{I\times 1}>0$ is the drift coefficient and the positive semidefinite matrix $\bm{\sigma}_{I\times p}\geq 0$ is the diffusion coefficient and 
\[
\bm{\sigma}[u,\bm{\sigma}_2^*,\mathbf{W}(u),\mathbf{Z}(u,w)]=
\bm{\sigma}_1[u,\mathbf{W}(u),\mathbf{Z}(u,w)]+\bm{\sigma}_2^*. 
\]
It is important to note that $\bm{\sigma}_1$ comes from the weather conditions of the venue of the game, percentage attendance of the home crowd and type of match (i.e. day or day-night match), and $\bm{\sigma}_2^*$ comes from the behavior of the bowler of the opposition team coming from the fractal dimensional strategy space and is measured by the square root of the product of the complex 
characteristic function $\Phi_{\hat g_U}(\theta)$ with its conjugate $\overline{\Phi}_{\hat g_U}(\theta)$  which will be discussed in Lemma \ref{l2}.

\begin{as}\label{as0}
	For $U>0$, let $\hat{\bm{\mu}}(u,\mathbf{W},\mathbf{Z}):[0,U]\times \mathbb{R}^{I}\times \mathbb{R}^I \ra\mathbb{R}^I$ and $\bm{\sigma}(u,\bm{\sigma}_2^*,\mathbf{W},\mathbf{Z}):[0,U]\times \mathbb{S}^{I\times U}\times \mathbb{R}^{I}\times \mathbb{R}^I \ra\mathbb{R}^I$ be some measurable function with $I\times U$-dimensional two-sphere $\mathbb{S}^{I\times U}$ and, for some positive constant $K_1$, $\mathbf{W}\in\mathbb{R}^I$ and, $\mathbf{Z}\in\mathbb{R}^I$ we have linear growth as
\[
|\hat{\bm{\mu}}(u,\mathbf{W},\mathbf{Z})|+
|\bm{\sigma}(u,\bm{\sigma}_2^*,\mathbf{W},\mathbf{Z})|\leq 
K_1(1+|\mathbf{Z}|),
\]
such that, there exists another positive, finite, constant $K_2$ and for a different score vector 
$\tilde{\mathbf{Z}}_{I\times 1}$ such that the Lipschitz condition,
\[
|\hat{\bm{\mu}}(u,\mathbf{W},\mathbf{Z})-
\hat{\bm{\mu}}(u,\mathbf{W},\widetilde{\mathbf{Z}})|+|\bm{\sigma}(u,\bm{\sigma}_2^*,
\mathbf{W},\mathbf{Z})-\bm{\sigma}(u,\bm{\sigma}_2^*,\mathbf{W},\widetilde{\mathbf{Z}})|
\leq K_2\ |\mathbf{Z}-\widetilde{\mathbf{Z}}|,\notag
\]
$ \widetilde{\mathbf{Z}}\in\mathbb{R}^I$ is satisfied and
\[
|\hat{\bm{\mu}}(u,\mathbf{W},\mathbf{Z})|^2+
\|\bm{\sigma}(u,\bm{\sigma}_2^*,\mathbf{W},\mathbf{Z})\|^2\leq K_2^2
(1+|\widetilde{\mathbf{Z}}|^2),
\]
where 
$\|\bm{\sigma}(u,\bm{\sigma}_2^*,\mathbf{W},\mathbf{Z})\|^2=
\sum_{i=1}^I \sum_{j=1}^I|{\sigma^{ij}}(u,\bm{\sigma}_2^*,\mathbf{W},\mathbf{Z})|^2$.
\end{as}

\begin{as}\label{as1}
	There exists a probability space $(\Omega,\mathcal{F}_u^{\mathbf Z},\mathcal{P})$ with sample space $\Omega$, filtration at $u^{th}$ over of run ${\mathbf{Z}}$ as $\{\mathcal{F}_u^{\mathbf{Z}}\}\subset\mathcal{F}_u$, a probability measure $\mathcal{P}$ and a $p$-dimensional $\{\mathcal{F}_u\}$ Brownian motion $\mathbf{B}$ where the measure of valuation of players $\mathbf{W}$ is an $\{\mathcal{F}_u^{\mathbf{Z}}\}$ adapted process such that Assumption \ref{as0} holds, for the feedback control measure of players there exists a measurable function $h$ such that $h:[0,U]\times C([0,U]):\mathbb{R}^I\ra\mathbb{W}$ for which $\mathbf{W}(u)=h[\mathbf{Z}(u,w)]$ such that Equation (\ref{1})	has a strong unique solution \citep{ross2008}.
\end{as}
As at the beginning of each innings each team starts with a zero score hence, initial condition is $\mathbf{Z}_{I\times 1}^*=\mathbf{0}_{I\times 1}$. Furthermore, we assume $\mathbf{W}$ is a Markov control. Therefore, there exists a measurable function $h:[0,U]\times C([0,U]:\mathbb{R}^I)\ra \mathbb{W}$ such that $\mathbf{W}(u)=h[\mathbf{Z}(u,w)]$. That is,  in order to know $\mathbf{W}$ we need to know $\mathbf{Z}$ first and it cannot be exogenously specified.
 
 \begin{definition}\label{de0}
 Suppose $\mathbf{Z}(u,w)$ is a non-homogeneous Fellerian semigroup on overs in $\mathbb{R}^I$. The infinitesimal generator $A$ of $\mathbf{Z}(u,w)$ is defined by,
\[
Ah(z)=\lim_{u\ra 0}\frac{\E_u[h(\mathbf{Z}(u,w))]-h(z(w))}{u},
\]
 for $z\in\mathbb{R}^I$	where $h:\mathbb{R}^I\ra\mathbb{R}$ is a $C_0^2(\mathbb{R}^I)$ function, $\mathbf{Z}$ has a compact support, and at $z(w)$ the limit exists where $\E_u$ represents team $T$'s conditional expectation of run $\mathbf{Z}$ at over $u$. Furthermore, if the above Feller semigroup is homogeneous on overs, then $Ah$ is exactly equal to the Laplace operator.
\end{definition}
  
As $h$ is a measurable function depending on $u$, there is a possibility that this function might have very large values and may be unstable. In order to stabilize $\mathbf{W}$ we need to take the natural logarithmic transformation and define a characteristic like quantum operator as in Definition \ref{de1}. 
\begin{defn}\label{de1}
For a Fellerian semigroup $\mathbf{Z}(u,w)$ for a small over interval $[u,u+\epsilon]$ with $\epsilon\downarrow 0$, define a characteristic-like quantum operator where the process starts at $u$ is defined as 
\[
\mathcal{A} h(z)=\lim_{\epsilon\ra 0}
\frac{\log\E_u[\epsilon^2\ h(\mathbf{Z}(u,w))]-\log[\epsilon^2h(z(w))]}{\log\E_u(\epsilon^2)},
\]
for $z\in\mathbb{R}^I$,	where $h:\mathbb{R}^I\ra\mathbb{R}$ is a $C_0^2(\mathbb{R}^I)$ function, $\E_u$ represents the conditional expectation of run $\mathbf{Z}$ at $u^{th}$ over,  for $\epsilon>0$ and a fixed $h$ we have the sets of all open balls of the form $B_\epsilon(h)$ contained in $\mathcal{B}$ (set of all open balls) and as $\epsilon\downarrow 0$ then $\log\E_u(\epsilon^2)\ra\infty$.
\end{defn} 

\begin{lem}\label{l0}(Dynkin formula) 
Suppose a Feller process $\mathbf{Z}(u,w)$ follows  Assumptions \ref{as0}, \ref{as1}, Definition \ref{de1} and Equation (\ref{1}). For $h\in C_0^2(\mathbb{R}^I)$ with the over interval $[\nu,\tilde{\nu}]$ and $\E_\nu[\tilde{\nu}]<\infty$ we have,
	
\begin{multline}\label{d8}
\log \E_{\tilde\nu}\left\{{\tilde\nu}^2 h(\mathbf{Z}_{\tilde{\nu}})\right\} =  
\log[\nu^2h(\mathbf{Z}_\nu)] 
 +\log\left[1+\frac{1}{\nu^2h(\mathbf{Z}_\nu)}\times\right. \\
 \left.
 \E_\nu\left\{ \int_{\nu}^{\tilde{\nu}} u^2\left(\sum_{i=1}^I\ \hat{\mu}_i
\frac{\partial h(\mathbf{Z}) }{\partial Z_i}
 +\mbox{$\frac{1}{2}$}
 \sum_{i=1}^I\sum_{j=1}^I (\sigma\sigma^T)_{ij} 
\frac{\partial^2 h(\mathbf{Z})}{\partial Z_i \partial Z_j}\right)  du\right\}\right],
\end{multline}
where $\E_\nu$ is the conditional expectation of the run at the beginning of the over interval, $\nu^2 h(\mathbf{Z}_\nu)\neq 0$, and with respect to the probability law $R^\nu$ for $\mathbf{Z}(u,w)$ starting at $\mathbf{Z}_\nu$ we have that
\[
R^\nu[\mathbf{Z}_{\nu_1}\in F_1,...,\mathbf{Z}_{\nu_m}\in F_m]=
P^0[\mathbf{Z}_{\nu_1}^\nu\in F_1,...,\mathbf{Z}_{\nu_m}^\nu\in F_m],
\]
where the $F_i$'s are Borel sets.
\end{lem}
\begin{proof}
Suppose the measure of valuation of all the players in team $T$ is $\mathbf{W}(u)=h[\mathbf{Z}(u,w)]$ where $h$ is a $C_0^2(\mathbb{R}^I)$ function. Here Markov control $\mathbf{W}(u)$ can be written in terms of $\mathbf{Z}$ because, if player $i$ scores more runs then his reputation will be higher as a batsman and $\mathbf{W}(u)$ will take very high value. Hence, instead of abusing of notations we can directly say that, $\mathbf{W}(u)=h(\mathbf{Z})$. For all $Z_i\in\mathbf{Z}$, where $i=1,2,...,I$, applying It\^o's formula on $\mathbf{W}(u)$ yields,
\begin{multline}\label{d1}
d\mathbf{W}(u)=
\sum_{i=1}^I\hat{\mu}_i
\frac{\partial h(\mathbf{Z})}{\partial Z_i}du+
\sum_{i=1}^I
\frac{\partial h(\mathbf{Z})}{\partial Z_i}(\sigma dB)_i\\
+\mbox{$\frac{1}{2}$}
\sum_{i=1}^I\sum_{j=1}^I\frac{\partial^2 h(\mathbf{Z})}{\partial Z_i \partial Z_j}(\sigma\sigma^T)_{ij}du.
\end{multline}	
After using the integral form of $(\sigma dB)_i\ (\sigma dB)_j=(\sigma\sigma^T)_{ij}\ du$, Equation (\ref{d1}) multiplied by $u^2$ is
\begin{multline}\label{d2}
u^2h(\mathbf{Z}_U)=
u^2h(\mathbf{Z}_0)+ \\
\int_{0}^Uu^2 
\left(\sum_{i=1}^I\hat\mu_i
\frac{\partial h(\mathbf{Z})}{\partial Z_i}+
\mbox{$\frac{1}{2}$}\sum_{i=1}^I\sum_{j=1}^I\ (\sigma\sigma^T)_{ij}
\frac{\partial^2h(\mathbf{Z})}{\partial Z_i \partial Z_j}\right)du\\
+\int_{0}^U u^2\sum_{i=1}^I
\frac{\partial h(\mathbf{Z})}{\partial Z_i}(\sigma dB)_i.
\end{multline}
Subdivide the entire $[0,U]$ over one-day match into small over-intervals $[\nu,\nu+\epsilon]$ such that $\epsilon\downarrow 0$ and define $\nu+\epsilon=\tilde \nu$. Therefore, 
\begin{multline*}
\E_{\tilde{\nu}}\left\{{\tilde{\nu}}^2 h(\mathbf{Z}_{\tilde{\nu}})\right\}=\nu^2h(\mathbf{Z}_\nu)\\
+\E_\nu\left\{ \int_{\nu}^{\tilde{\nu}}u^2
\left(\sum_{i=1}^I\hat\mu_i\frac{\partial h(\mathbf{Z})}{\partial Z_i}+
\mbox{$\frac{1}{2}$}\ \sum_{i=1}^I\sum_{j=1}^I
(\sigma\sigma^T)_{ij}\frac{\partial^2h(\mathbf{Z})}{\partial Z_i \partial Z_j}\right)du\right\}\\ 
+\E_\nu\left\{\int_{\nu}^{\tilde{\nu}}u^2\sum_{i=1}^I
\frac{\partial h(\mathbf{Z})}{\partial Z_i}(\sigma dB)_i\right\}.
\end{multline*}
Taking natural logarithm in both the sides yields,
\begin{multline*}
\log\E_{\tilde \nu}[\tilde{\nu}^2 h(\mathbf{Z}_{\tilde{\nu}})]=\log[\nu^2h(\mathbf{Z}_\nu)]+\\
\log\left[1+\frac{1}{\nu^2h(\mathbf{Z}_\nu)}
\E_\nu\left[ \int_{\nu}^{\tilde{\nu}} u^2
\left(\sum_{i=1}^I\hat\mu_i\frac{\partial h(\mathbf{Z})}{\partial Z_i}+
\mbox{$\frac{1}{2}$}\sum_{i=1}^I\sum_{j=1}^I(\sigma\sigma^T)_{ij}
\frac{\partial^2h(\mathbf{Z})}{\partial Z_i \partial Z_j}\right)du\right]\right. \\
\left.+\frac{1}{\nu^2h(\mathbf{Z}_\nu)}
\E_\nu\left\{\int_{\nu}^{\tilde{\nu}} u^2\sum_{i=1}^I
\frac{\partial h(\mathbf{Z})}{\partial Z_i}(\sigma dB)_i\right\}\right].
\end{multline*}
Assume $K(\mathbf{Z}_{u})=u^2\kappa(\mathbf{Z}_{u})$ is a bounded Borel measurable function where $\sigma u^2 \frac{\partial}{\partial Z}h(\mathbf Z_u)\leq u^2\kappa(\mathbf{Z}_{u})$ and for a finite $M$ we have $|K(\mathbf{Z}_{u})|\leq M$. For all integers $m$ and a simple function $\phi_{[u<\tilde{\nu}]}$ we have,
\begin{multline}\label{d5}
\E_\nu\left[\int_{\nu}^{\tilde{\nu}\wedge m}K(\mathbf{Z}_{u})dB_u\right]\leq 
\E_\nu\left[\int_{0}^{\tilde{\nu}\wedge m}K(\mathbf{Z}_{u})dB_u\right]\\
=\E_\nu\left[\int_{0}^{ m}\phi_{[u<\tilde{\nu}]} K(\mathbf{Z}_{u})dB_u\right]=0,
\end{multline}
as $\phi_{[u<\tilde{\nu}]}$ and $K(\mathbf{Z}_{u})$ are $\mathcal{H}_u$-measurable where $\mathcal{H}_u$ is the $\sigma$-algebra generated by $u$. Moreover,
\begin{multline}\label{d6}
\E_\nu\left[\left(\int_{\nu}^{\tilde{\nu}\wedge m}K(\mathbf{Z}_{u})dB_u\right)^2\right]\leq \E_\nu\left[\left(\int_{0}^{\tilde{\nu}\wedge m}K(\mathbf{Z}_{u})dB_u\right)^2\right]\\
=\E_\nu\left[\int_{0}^{\tilde{\nu}\wedge m}K^2(\mathbf{Z}_{u})du\right]\leq M^2\E_\nu[\tilde{\nu}]<\infty.
\end{multline}
As at the beginning of the over interval the batsman does not know what kind of bowl is going to be delivered, their conditional expectation with respect to the run at $\nu$ should be same as at $\tilde{\nu}$. Hence, in Equation (\ref{d6}) we assume $\E_\nu[{\nu}]=\E[\tilde\nu]$. From the above argument we can say that, $\left\{\int_{\nu}^{\tilde{\nu}\wedge m}K(\mathbf{Z}_{u})\ dB_u\right\}_{m>0}$ is uniformly integrable with respect to the probability law stating at $\mathbf{Z}_\nu$ and is defined as $R^\nu$. Finally, taking limit with respect to $m$, Equation (\ref{d5}) becomes,
\begin{multline*}
0=\lim_{m\rightarrow\infty}
\E_\nu\left[\int_{\nu}^{\tilde{\nu}\wedge m}K(\mathbf{Z}_u)dB_u\right]\\
=\E_\nu\left[\lim_{m\ra\infty}\int_{\nu}^{\tilde{\nu}\wedge m}K(\mathbf{Z}_u)dB_u\right]= 
\E_\nu\left[\int_{\nu}^{\tilde{\nu}}K(\mathbf{Z}_{u})dB_u\right],
\end{multline*}
so that
\begin{multline*}
\log\E_{\tilde\nu}[{\tilde\nu}^2 h(\mathbf{Z}_{\tilde{\nu}})]=
\log[\nu^2h(\mathbf{Z}_\nu)]\\
+\log\left[1+\frac{1}{\nu^2h(\mathbf{Z}_\nu)}
\E_\nu\left\{ \int_{\nu}^{\tilde{\nu}} u^2
\left(\sum_{i=1}^I\hat{\mu}_i\frac{\partial h(\mathbf{Z})}{\partial Z_i}+ 
\right.\right.\right.\\
\left.\left.\left.
+\mbox{$\frac{1}{2}$}
\sum_{i=1}^I\sum_{j=1}^I\ (\sigma\sigma^T)_{ij}
\frac{\partial^2h(\mathbf{Z})}{\partial Z_i \partial Z_j}\right)  du\right\}\right].
\end{multline*}
\end{proof}

From \cite{oksendal2003} we know the infinitesimal generator, 
\[
Ah(z)=\sum_{i=1}^I \hat \mu_i\frac{\partial}{\partial Z_i}h(\mathbf{Z}) + 
\frac{1}{2} \sum_{i=1}^I\sum_{j=1}^I 
(\sigma\sigma^T)_{ij} \frac{\partial^2h(\mathbf{Z})}{\partial Z_i \partial Z_j}.
\] 
Therefore, Equation (\ref{d8}) becomes,
\begin{multline*}
\log\E_{\tilde\nu}[{\tilde\nu}^2 h(\mathbf{Z}_{\tilde{\nu}})]=
\log\left[\nu^2h(\mathbf{Z}_\nu)\right]\\
+\log\left[1+\frac{1}{\nu^2h(\mathbf{Z}_\nu)}
\E_\nu\left\{ \int_{\nu}^{\tilde{\nu}} u^2\ Ah(z)  du\right\}\right].
\end{multline*}
If $Ah(z)=0$ then, Equation (\ref{1}) is a run trap.

\begin{lem}\label{l1}
Suppose $h$ is a $C_0^2(\mathbb{R}^I)$ function. 
Then for a small over-interval $[\nu,\tilde{\nu}]$ with $\epsilon\downarrow 0$ and 
$h(\mathbf{Z}_\nu)\neq 0$, 
\begin{equation}\label{d10}
\mathcal{A}h(z)=
\log\left[1+\frac{1}{h(\mathbf{Z}_\nu)}
\left(\sum_{i=1}^I\hat\mu_i\frac{\partial h(\mathbf{Z})}{\partial Z_i}+ 
\mbox{$\frac{1}{2}$}\sum_{i=1}^I\sum_{j=1}^I(\sigma\sigma^T)_{ij}
\frac{\partial^2h(\mathbf{Z})}{\partial Z_i \partial Z_j}\right)\right].
\end{equation}
\end{lem}

\begin{proof}
The terms inside the first bracket in right hand side of Equation (\ref{d10}) can be replaced by $Ah(z)$ \citep{oksendal2003}. If $z\in\mathbb{R}^I$ is a run trap, then $Ah(z)=0$ and which implies $\mathcal{A}h(z)=\log 1=0$. Consider an open and bounded set $\mathcal{D}_0$ such that $z\in\mathcal{D}_0$ and fit our $h$ function outside $\mathcal{D}_0$. As $h$ is a $C_0^2(\mathbb{R}^I)$ function $\mathcal{A}(h)=Ah(z)=0$. If $z\in\mathbb{R}^I$ is not a trap. Then consider a bounded open set $z\in\mathcal{D}_1$ such that $\E_\nu[\nu^2]<\infty$. Using Lemma \ref{l0} it can be shown that,
\begin{multline}\label{d11}
\lim_{\epsilon\ra 0}
\left|\frac{\log\E_{\tilde\nu}[{\tilde\nu}^2 h(\mathbf{Z}_{\tilde{\nu}})]-
\log[\nu^2h(\mathbf{Z}_\nu)]}
{\log\E_\nu[\nu^2]}-
\frac{\log[1+\nu^2Ah(z)]}{\nu^2\ h(z_\nu)}\right|=\\ 
\lim_{\epsilon\ra 0}
\left|\frac{\E_\nu\log\left[1+\int_{\nu}^{\tilde\nu}u^2 Ah(\mathbf{Z})du\right]-
\E_\nu \log\left[1+\int_{\nu}^{\tilde{\nu}}u^2Ah(\mathbf{Z})du\right]}
{\nu^2h(z_\nu)\log\E_\nu[\nu^2]}\right|=\\ 
\lim_{\epsilon\ra 0}\ \left|\frac{\E_\nu \log 1}{\nu^2h(z_\nu)\log\E_\nu[\nu^2]}\right|\leq
\lim_{\epsilon\ra 0}\ \sup_{w\in\mathcal{D}_1}\left|Af(z)-Af(w)\right|= 0,
\end{multline}
for $|z-w|<|\xi|$,
where for a finite positive number $\eta$ define $|\xi|\leq\eta\epsilon[\mathbf{Z}^T]^{-1}$. The inequality in (\ref{d11}) holds because, both the natural logarithm and $Ah$ operator{\tiny } are continuous. 
\end{proof}

In Equation (\ref{1}) the stochastic drift component $\hat{\bm{\mu}}$ is replaced by a deterministic function $\bm{\mu}$ such that 
\[
\bm{\mu}[u,\mathbf{W}(u),\mathbf{Z}(u,w)]=\mathcal{A}h(z)+\widetilde{\bm{\mu}}[u,\mathbf{W}(u),\mathbf{Z}(u,w),f(\mathbf{Z}_u)],
\]
where $\mathcal{A}h$ is a  characteristic like quantum operator of a Feller semigroup, $f(\mathbf{Z}_u)$ is the probability distribution of $\mathbf{Z}(u,w)$ and for a given over $u$, $k_1>0$, $k_2>0$ and $l>0$ we have $\widetilde{\bm{\mu}}:[0,U]\times\mathbb{R}^I\times\mathbb{R}^I\times(N_{\gamma^2}(\mathbb{R}^I),\rho) \ra\mathbb{R}^I$ that satisfies following conditions:	
\[
||\widetilde{\bm\mu}(u,\mathbf{W},\mathbf{Z},f)-\widetilde{\bm\mu}(u,\widehat{\mathbf{W}},\widehat{\mathbf{Z}},\hat f)||\leq k_1(||\mathbf{W}-\widehat{\mathbf W}||)+k_2(||\mathbf{Z}-\widehat{\mathbf Z}||)+\rho(f,\hat f),
\]
and
\[
||\widetilde{\bm \mu}(u,\mathbf{W},\mathbf{Z},f)||\leq l(1+||\mathbf W||+||\mathbf Z||+||f||_\gamma),
\]
where $\gamma(z)\equiv 1+||z||$, $(N_{\gamma^2}(\mathbb{R}^I),\rho)$ forms a complete metric space with the metric $\rho(f,\hat f)$ \citep{govindan2016}.

Now Assume the diffusion component $\bm{\sigma}(u,\bm{\sigma}_2^*,\mathbf{W},\mathbf{Z})$ is additively separable into $\bm{\sigma}_1 (u,\mathbf{W},\mathbf{Z})$ and $\bm{\sigma}_2^*$, where $\bm{\sigma_2}^*=[\Phi_{\hat g_U}(\theta)* \overline{\Phi}_{\hat g_U}(\theta)]^{1/2}$,
with $\Phi_{\hat g_U}(\theta)$ being the characteristic function defined in Lemma \ref{l2} below. $\bm{\sigma}_1 (u,\mathbf{W},\mathbf{Z})$ consists of the venue of the game, the percentage of attendance of the home crowd, the type of one-day match, the amount of dew on the pitch and the speed of wind. 

Firstly, if team $T$ is playing abroad, players have harder time scoring a run than in their home and thinking of this they create extra mental strains on themselves. We assume that pressure $\mathbf{p}$ is a non-negative $C^2$ function $\mathbf{p}(u,\mathbf{Z}):[0,U]\times\mathbb{R}^I\ra\mathbb{R}_+^I$ at match $M+1$ such that if $\mathbf{Z}_{u-1}<\E_{u-1}(\mathbf{Z})$ then $\mathbf{p}$ takes a very high positive value. Secondly, We define attendance rate as a positive finite $C^2$ function $\mathbf{A}(u,\mathbf{W}):[0,U]\times\mathbb{R}^I\ra\mathbb{R}_+^I$ with $\partial \mathbf{A}/\partial\mathbf{W}>0$ and $\partial \mathbf{A}/\partial u\gtreqless 0$ depends on if at over $u$, player $i$ with valuation $W_i\in\mathbf{W}$ is still playing, or is out. Thirdly, assume the effect of day or day-night one-day match be a function $\mathfrak{B}(\mathbf{Z})\in\mathbb{R}_+^I$ such that,
\begin{align}\label{m0}
\mathfrak{B}(\mathbf{Z})=\mbox{$\frac{1}{2}$} 
[\mbox{$\frac{1}{2}$} \E_0(\mathbf{Z}_D^2)+ \mbox{$\frac{1}{2}$}\ \E_0(\mathbf{Z}_{DN}^1)]+
\mbox{$\frac{1}{2}$} 
[\mbox{$\frac{1}{2}$} \E_0(\mathbf{Z}_D^1)+\mbox{$\frac{1}{2}$} \E_0(\mathbf{Z}_{DN}^2)],
\end{align}
where for $i=1,2$, $\E_0(\mathbf{Z}_D^i)$ is the conditional expectation of run of team $T$ before the starting of the day match $M+1$ with run at the $i^{th}$ innings $\mathbf{Z}_D^i$, and $\E_0(\mathbf{Z}_{DN}^i)$ is the conditional expectation of the run before starting a day-night match $M+1$. Furthermore, if team $T$ wins the toss, then it will go for the payoff $\mbox{$\frac{1}{2}$} \left[\mbox{$\frac{1}{2}$} \E_0(\mathbf{Z}_D^2)+ \mbox{$\frac{1}{2}$} \E_0(\mathbf{Z}_{DN}^1)\right]$, and the later part of the Equation (\ref{m0}) otherwise. Finally, as the amount of dew on grass and the speed of wind at over $u$ are ergodic, following \cite{falconer2004}  we assume this can be represented by a Weierstrass function $\mathbf{Z}_e:[0,U]\ra\mathbb{R}$ defined as,
\begin{equation}\label{m1}
\mathbf{Z}_e(u)=\sum_{\a=1}^{\infty}(\lambda_1+\lambda_2)^{(s-2)\a} \sin\left[(\lambda_1+\lambda_2)^\a u\right],
\end{equation}
where $s\in(1,2)$ is a penalization constant of weather at over $u$, $\lambda_1$ is the dew point measure and $\lambda_2$ is the speed of wind such that $(\lambda_1+\lambda_2)>1$.

\begin{as}\label{as2}
$\bm{\sigma}_1(u,\mathbf{W},\mathbf{Z})$ is a positive, finite part of the diffusion component in Equation (\ref{1}) which satisfies Assumptions \ref{as0} and \ref{as1} and is defined as 
\begin{multline}\label{m3}
\bm{\sigma}_1(u,\mathbf{W},\mathbf{Z})=\mathbf{p}(u,\mathbf{Z})+\mathbf{A}(u,\mathbf{W})+
\mathfrak{B}(\mathbf{Z})+\mathbf{Z}_e(u)\\
+\rho_1 \mathbf{p}^T(u,\mathbf{Z})\mathbf{A}(u,\mathbf{W})+
\rho_2\mathbf{A}^T(u,\mathbf{W})\mathfrak{B}(\mathbf{Z})+
\rho_3\mathfrak{B}^T(\mathbf{Z})\mathbf{p}(u,\mathbf{Z}),
\end{multline}
where $\rho_j\in(-1,1)$ is the $j^{th}$ correlation coefficient for $j=1,2,3$, and $\mathbf{A}^T,\mathfrak{B}^T$ and $\mathbf{p}^T$ are the transposition of $\mathbf{A},\mathfrak{B}$ and $\mathbf{p}$ which satisfy all conditions with Equations (\ref{m0}) and (\ref{m1}). As the ergodic function $\mathbf{Z}_e$ comes from nature, team $T$ does not have any control on it and its correlation coefficient with other terms in Equation (\ref{m3}) are assumed to be zero.
\end{as} 

 The randomness of comes from the delivery of the bowler of $\bm{\sigma}_2^*$ of Equation (\ref{1}) the opposition team. There are mainly two types of bowlers: pacers and spinners. Pacers  have two components, the speed of the ball $s\in\mathbb{R}_+^{I\times U}$ in miles per hour and the curvature of the bowl path measure by the dispersion from the straight line connecting two middle stumps measured by $x\in\mathbb{R}_+^{I\times U}$ inches. Define a payoff function $A_1(s,x,G):\mathbb{R}_+^{I\times U}\times \mathbb{R}_+^{I\times U}\times[0,1]\ra \mathbb{R}^{2I\times U}$ such that, at over $u$ the expected payoff after guessing a ball right is $\E_u A_1(s,x,G)$, where $G$ is a guess function such that, if the batsman guesses a bowler's delivery properly then $G=1$ and if he does not then $G=0$, and if he partially guesses then, $G\in(0,1)$.

On the other hand, there is a payoff function $A_2$ for leg spinner such that $A_2(s,x,\theta_1,G):\mathbb{R}_+^{I\times U}\times \mathbb{R}_+^{I\times U}\times(\pi/2,\pi]\times[0,1]\ra\mathbb{R}^{2I\times U} $, where $\theta_1$ is the angle between the line between the bowler's hand and its first drop on the crease and the line connecting the second point and the bat. The expected payoff to run at over $u$ when the bowler is a leg spinner is $\E_u A_2(s,x,\theta_1,G)$. Finally, an off spinner's payoff function is $A_3(s,x,\theta_1,\theta_2,G):\mathbb{R}_+^{I\times U}\times \mathbb{R}_+^{I\times U}\times(\pi/2,\pi]\times(0,\pi/36]\times[0,1]\ra\mathbb{R}^{2I\times U}$, where $\theta_2$ is the allowable elbow extension during a delivery with the expectation at over $u$ as $\E_u A_3(s,x,\theta_1,\theta_2,G)$. If $\theta_2$ is more than $\pi/36$, the off spinner gets extra spin to make a teesra. As a batsman does not know who is coming to bowl after a $6$-ball-over is completed, their total expected payoff function at over $u$ is $\mathcal{A}(s,x,\theta_1,\theta_2,G)=\wp_1\ \E_u A_1(s,x,G)+\wp_2\ \E_u A_2(s,x,\theta_1,G)+\wp_3\ \E_u A_3(s,x,\theta_1,\theta_2,G)$, where for $j\in\{1,2,3\}$, $\wp_j$ is the probability of each of a pacer, a leg and an off spinners with $\wp_1+\wp_2+\wp_3=1$. Therefore, $\mathcal{A}(s,x,\theta_1,\theta_2,G):\mathbb{R}_+^{I\times U}\times \mathbb{R}_+^{I\times U}\times(\pi/2,\pi]\times(0,\pi/36]\times[0,1]\ra\mathbb{R}^{2I\times U}$.

Delivering a ball is a realization from a fractal strategy space of the opposition bowler's mind to get a wicket which follows a quantum field theory \citep{kappen2007}. Therefore, we define the strategy space as a conformal field such that the plane  $\mathbb{C}^{I\times U}=\mathbb{R}^{2I\times U}$ is subset of the $I\times U$-dimensional two-sphere $\mathbb{S}^{I\times U}=\mathbb{C}^{I\times U}\cup \{\infty\}$, a one-point compactification of the plane $\mathbb{C}^{I\times U}$ \citep{schramm2011}. Let $\mathbb{U}:=\{\mathcal{A}\in\mathbb{C}^{I\times U}:\ |\mathcal{A}|<1\}$ be a unit sphere and for a compact simple path $\gamma\in \overline{\mathbb{U}}^{I\times U}\setminus \{\mathbf{0}\}^{I\times U}$ with its end point at $\gamma\cap\partial\mathbb{U}$ there exists a unique conformal homeomorphism $h_\gamma:\mathbb{U}^{I\times U}\ra\mathbb{U}^{I\times U}\setminus\gamma$ such that, $h_\gamma(0)=\mathbf{0}$ and $h_\gamma'(0)\in\mathbb{R}_+^{I\times U}$, where $h_\gamma '=\partial h/\partial \gamma$ \citep{schramm2011}. For another compact simple path $\hat{\gamma}\in\overline{\mathbb{U}}^{I\times U}$ assume the two end points are at $0$ and $\hat{\gamma}\cap\partial\mathbb{U}$ respectively. Suppose $\hat{\gamma}_r$ is the arc of the path joining $0$ and $\hat{\gamma}\cap\partial\mathbb{U}$ for each point $r\in\hat{\gamma}\setminus\{\mathbf{0}\}^{I\times U}$. If $\tilde h(r):=\log h_{\hat{\gamma}_r}'(0)$ then, $\tilde h$ is a homeomorphism from $\hat{\gamma}\setminus\{\mathbf{0}\}^{I\times U}$ onto $[0,U)$. Suppose $r(u)$ is the inverse map $r:[0,U)\ra\hat{\gamma}$, and set $h(\mathcal{A},u)=h_u(\mathcal{A}):=h_{\hat{\gamma}_{r(u)}}(\mathcal{A})$ and for another multivariate function in unit sphere $g_u(\mathcal{A})$ with $g_u(0)=\mathbf{0}^{I\times U}$ such that, $\phi_u(\mathcal{A})=g_u^{-1}(h_u(\mathcal{A}))$, where $\phi_u(\mathcal{A})$ is a multivariate mapping of the unit sphere $\mathbf{U}$. Loewner's Slit mapping theorem and Schramm-Loewner \citep{schramm2011}  type evolution equation gives
\begin{equation}\label{m4}
\frac{\partial}{\partial u}g_u(\mathcal{A})= 
g_u(\mathcal{A})\frac{\sqrt{\kappa\Upsilon_\eta(\mathcal{A})}\mathfrak{W}_u+
	g_u(\mathcal{A})}{\sqrt{\kappa\Upsilon_\eta(\mathcal{A})}\mathfrak{W}_u-
	g_u(\mathcal{A})},
\end{equation}
where $\kappa\in[0,\infty)$ is a positive diffusivity parameter, $\mathfrak{W}_u$ is a Brownian motion on two-sphere $\mathbb{S}^{I\times U}$ and 
$\Upsilon_\eta(\mathcal{A})\in(0,\infty)$ is the area of a von Koch snowflake curve with total number of iteration $\eta(\mathcal A)$ \citep{koch1906}. 

To construct the $\Upsilon_\eta(\mathcal A)$ function suppose that the total number of iterations is defined as  $\eta(\mathcal A):\mathbb{R}^{I\times U}\ra\mathbb{R}_+^{I\times U}\cup\{\mathbf{0}\}^{I\times U}$ such that $\partial\eta/\partial\mathcal{A}>0$ and $\partial^2\eta/\partial\mathcal{A}^2>0$. The main reason behind this assumption is that if the payoff $\mathcal{A}$ is high then a batsman can hit a $6$ from a delivery and then he needs to do more iterations of the von Koch snowflake strategy space of a bowler. Let us denote $\hat{\theta}_\eta$ as the number of sides, $\tilde{\theta}_\eta$ as the length of the each side, and $\Xi_\eta$ as the perimeter of the strategy space of the bowler before $u^{th}$ over. Then $\hat{\theta}_\eta=3*4^{\eta(\mathcal A)}$, $\tilde{\theta}_\eta=(1/3)^{\eta(\mathcal A)}$, and $\Xi_\eta=\hat{\theta}_\eta *\tilde{\theta}_\eta=3*(4/3)^{\eta(\mathcal A)}$, such that $\mathcal A=0$ implies $\eta(\mathcal A)=0$ and then $\hat{\theta}_\eta=3$ which means the strategy space becomes a triangle which we denote by $\Delta$. Therefore, by \cite{koch1906} the area of the von Koch snowflake $\Upsilon_\eta(\mathcal A)=\Upsilon_{\eta-1}(\mathcal A)+(1/3)(4/9)^{\eta(\mathcal A)}\Delta$ and finally, for $\Upsilon_0(\mathcal A)=\Delta$, we have $\Upsilon_\eta(\mathcal A)=\frac{\Delta}{5} \left[8-3(\frac{4}{9})^{\eta(\mathcal A)}\right]$ which results in the diffusion part of the Equation (\ref{m4}).

Now, for a fixed $\Upsilon_\eta$ define a shifted conformal map $\hat g_u(\mathcal{A})\equiv\sqrt{\kappa \Upsilon_\eta} \mathfrak{W}_u-g_u(\mathcal{A})$ such that, $\hat g_u^{-1}(\omega)\stackrel{d}{=} g_u^{-1}(\sqrt{\kappa \Upsilon_\eta}\ \mathfrak{W}_u-\omega)$, where $\stackrel{d}{=}$ indicates the equality of the distributions of the stochastic process \citep{najafi2015}. Therefore,
\[
\partial \hat{g}_u(\mathcal{A})= 
\sqrt{\kappa\Upsilon_\eta(\mathcal{A})}\partial\mathfrak{W}_u-\partial g_u(\mathcal{A}).
\]
Equation (\ref{m4}) implies,
\begin{equation}\label{m7}
\partial \hat g_u(\mathcal{A})=
-\frac{1}{\hat g_u(\mathcal A)}
\left[\sqrt{\kappa\Upsilon_\eta(\mathcal A)}\mathfrak{W}_u+
\hat g_u(\mathcal A)\right]^2\partial u
-\frac{\kappa\Upsilon_\eta(\mathcal A)\mathfrak{W}_u}{\hat g_u(\mathcal A)}\partial\mathfrak{W}_u,
\end{equation}
as two processes are generated from the same Brownian motion $\mathfrak W_u$, and where $\hat g_u(\mathcal A)\neq 0$.  As in both the drift and diffusion components of Equation (\ref{m7}) have the Brownian motion $\mathfrak{W}_u$, we define a function $\tilde g(\mathfrak W_u)=\mathfrak W_u$, and assuming it is a Feller process on two-sphere, define an infinitesimal generator $\mathcal L\tilde g_u$ on it such that, 
\begin{equation}\label{m8}
\partial\hat{g}_u(\mathcal{A})=
-\frac{1}{\hat g_u(\mathcal A)}
\left[\sqrt{\kappa\Upsilon_\eta(\mathcal A)}\mathcal{L}\tilde g_u+
\hat{g}_u(\mathcal A)\right]^2\partial u
-\frac{\kappa\Upsilon_\eta(\mathcal A)\mathcal{L}\tilde g_u}{\hat g_u(\mathcal A)}\partial \mathfrak{W}_u.
\end{equation}

\begin{lem}\label{l2}
Suppose $\Bbbk(\hat g)=\exp(\imath \theta \hat g )$ is a $C^2(\mathbb{S}^{I\times U})$ function and let $\hat g_u$ satisfies the stochastic differential equation specified in Equation (\ref{m8}). If the unique bounded function $\ell:[0,U]\times\mathbb{S}^{I\times U}\ra\mathbb{S}^{I\times U}$ satisfies the partial differential equation
\begin{multline*}
\mathcal{L}\ell(u,\hat g)=
\frac{\partial}{\partial u}\ell(u,\hat g)-
\frac{1}{\hat g(\mathcal A)}
\left[\sqrt{\kappa\Upsilon_\eta(\mathcal A)}\mathcal{L}\tilde{g}+\hat{g}(\mathcal A)\right]^2
\frac{\partial}{\partial \hat g}\ell(u,\hat g)\\
-\mbox{$\frac{1}{2}$}\left(\frac{\kappa\Upsilon_\eta(\mathcal A)\mathcal{L}\tilde g_u}{\hat g}\right)^2\frac{\partial^2}{\partial \hat g^2}\ell(u,\hat g)=0,
\end{multline*}
for all $u\in[0,U]$ and $\hat g\in\mathbb{S}^{I\times U}$
with terminal condition $\ell(U,\hat g)=\Bbbk(\hat g)$ given by $\ell(u,\hat g)=\E[\Bbbk(\hat g_U)|\hat g_u=\hat g]$ then, the characteristic function is
\begin{eqnarray*}
\Phi_{\hat g_U}(\theta) & = & 
\exp\left\{\imath\theta\exp\left[-\frac{1}{\hat g^2(\mathcal A)}\left(\sqrt{\kappa\Upsilon_\eta(\mathcal A)}\mathcal{L}\tilde{g}+
\hat{g}(\mathcal A)\right)^2U\right]\right.\\
& & -
\mbox{$\frac{1}{4}$}\theta^2
\left(\frac{\kappa\Upsilon_\eta(\mathcal A)\mathcal{L}\tilde g_u}{\sqrt{\kappa\Upsilon_\eta(\mathcal A)}\mathcal{L}\tilde{g}+\hat{g}}\right)^2\times \\
& & \left.\left[\exp\left\{-\frac{2}{\hat g^2(\mathcal A)}
\left(\sqrt{\kappa\Upsilon_\eta(\mathcal A)}\mathcal{L}\tilde{g}+
\hat{g}(\mathcal A)\right)^2U\right\}-1\right] \right\},
\end{eqnarray*}
where $\imath$ is an imaginary number and $\theta\in\mathbb{R}$.
\end{lem}

\begin{proof}
Assume, $\Bbbk(\hat g)=\exp(\imath \theta \hat g )$, $\theta\in\mathbb{R}$ and $\hat g\in \mathbb{S}^{I\times U}$. By the Feynman-Kac Representation Theorem, we know, 
\[
\ell(u,\hat g)=
\E[\Bbbk(\hat g_U)|\hat g_u=\hat g]=
\E[\exp(\imath\theta\hat g_U)|\hat g_u=\hat g]
\] 
is the unique bounded solution of the backward parabolic partial differential equation
\begin{multline}\label{m11}
0=
\frac{\partial}{\partial u}\ell(u,\hat g)-\frac{1}{\hat g}
\left[\sqrt{\kappa\Upsilon_\eta(\mathcal A)}\mathcal{L}\tilde{g}+\hat{g}\right]^2
\frac{\partial}{\partial \hat{g}}\ell(u,\hat g)\\
-\mbox{$\frac{1}{2}$} 
\left(\frac{\kappa\Upsilon_\eta(\mathcal A)\mathcal{L}\tilde g_u}{\hat g}\right)^2
\frac{\partial^2}{\partial \hat g^2}\ell(u,\hat g),
\end{multline}
for all $u\in[0,U]$
with terminal condition of the match $M+1$ as $\ell(U,\hat g)=\Bbbk(\hat g)=\exp(\imath\theta\hat g)$ for all $\hat g\in\mathbb{S}^{I\times U}$. As the characteristic function of $\hat g_U$ is  $\Phi_{\hat g_U}(\theta)=\ell(0,\hat g)=\E[\exp(\imath\theta\hat g)|\hat g_0=\hat g]$. Assume, $\ell$ is a $C^2$ function such that, $\ell(u,\hat g)=\exp\{\imath\theta\a(u)+\be(u)\}$, where at the final over $\a(U)=1$ and $\be(U)=0$. Now,
\[
\frac{\partial}{\partial u}\ell(u,\hat g)
=\left[\imath\theta\hat g\frac{\partial\a(u)}{\partial u}+
\frac{\partial\be(u)}{\partial u}\right]\ell(u,\hat g),
\]
\[
\frac{\partial}{\partial \hat{g}}\ell(u,\hat g)
=\imath\theta\a(u)\ell(u,\hat g),
\]
and
\begin{equation}\label{m12}
\frac{\partial^2}{\partial \hat g^2}\ell(u,\hat g)=-\theta^2\a^2(u)\ell(u,\hat g).
\end{equation}
The results of Equations (\ref{m11})-(\ref{m12}) imply,
\begin{multline}\label{m13}
\imath\theta\left[\hat{g}
\frac{\partial \a(u)}{\partial u}-\frac{1}{\hat g}
\left(\sqrt{\kappa\Upsilon_\eta(\mathcal A)}\mathcal{L}\tilde{g}+\hat{g}\right)^2\a(u)\right]\\
+\mbox{$\frac{1}{2}$}\theta^2\a^2(u)\left(\frac{\kappa\Upsilon_\eta(\mathcal A)\mathcal{L}\tilde g_u}{\hat g}
\right)^2+\frac{\partial\be(u)}{\partial u}=0.
\end{multline}
In order to maintain the zero right hand side condition of Equation (\ref{m13}) for each over $u\in[0,U]$ we must have,
\begin{equation}\label{m14}
\frac{\partial \a(u)}{\partial u}=\frac{1}{\hat g^2}
\left(\sqrt{\kappa\Upsilon_\eta(\mathcal A)}\mathcal{L}\tilde{g}+\hat g\right)^2\a(u)
\end{equation}
and
\[
\frac{\partial\be(u)}{\partial u}=-
\mbox{$\frac{1}{2}$}\theta^2\a^2(u)
\left(\frac{\kappa\Upsilon_\eta(\mathcal A)\mathcal{L}\tilde g_u}{\hat g}\right)^2.
\]
Solving the differential equation in Equation (\ref{m14}) with the final over condition 
$\a(U)=1$ yields,
\begin{equation}\label{m16}
\a(u)=
\exp\left\{-\frac{1}{\hat g^2}
\left(\sqrt{\kappa\Upsilon_\eta(\mathcal A)}
\mathcal{L}\tilde{g}+\hat{g}\right)^2(U-u)\right\}.
\end{equation}
Hence,
\[
\frac{\partial\be(u)}{\partial u}=-
\mbox{$\frac{1}{2}$}\theta^2
\left(\frac{\kappa\Upsilon_\eta(\mathcal A)\mathcal{L}\tilde g_u}{\hat g}\right)^2
\exp\left[-\frac{2}{\hat g^2}
\left(\sqrt{\kappa\Upsilon_\eta(\mathcal A)}\mathcal{L}\tilde g+\hat g\right)^2(U-u)\right],
\]
with the integral equation for the over interval $[0,u]$ given by,
\begin{multline*}
\be(u)-\be(0)=-
\mbox{$\frac{1}{4}$}\theta^2
\left(\frac{\kappa\Upsilon_\eta(\mathcal A)\mathcal{L}\tilde g_u}{\sqrt{\kappa\Upsilon_\eta(\mathcal A)}\mathcal{L}\tilde{g}+\hat{g}}\right)^2 
\times \\
\left[\exp\left\{-\frac{2}{\hat g^2}
\left(\sqrt{\kappa\Upsilon_\eta(\mathcal A)}
\mathcal{L}\tilde g+\hat g\right)^2(U-u)\right\}\right.\\ 
\left.-\exp\left\{-\frac{2}{\hat g^2}
\left(\sqrt{\kappa\Upsilon_\eta(\mathcal A)}
\mathcal{L}\tilde{g}+\hat{g}\right)^2U\right\}\right].
\end{multline*}
The terminal condition $\be(U)=0$ implies,
\begin{multline*}
\be(0)=-
\mbox{$\frac{1}{4}$}\theta^2
\left(\frac{\kappa\Upsilon_\eta(\mathcal A)\mathcal{L}\tilde g_u}{\sqrt{\kappa\Upsilon_\eta(\mathcal A)}\mathcal{L}\tilde{g}+\hat{g}}\right)^2
\times \\
\left[\exp\left\{-\frac{2}{\hat g^2}
\left(\sqrt{\kappa\Upsilon_\eta(\mathcal A)}\mathcal{L}\tilde g+\hat g\right)^2U\right\}-1\right],
\end{multline*}
and therefore
\begin{multline}\label{m21}
\be(u)=-
\mbox{$\frac{1}{4}$}\theta^2
\left(\frac{\kappa\Upsilon_\eta(\mathcal A)\mathcal{L}\tilde g_u}{\sqrt{\kappa\Upsilon_\eta(\mathcal A)}\mathcal{L}\tilde{g}+\hat{g}}\right)^2
\times \\
\left[\exp\left\{-\frac{2}{\hat g^2}
\left(\sqrt{\kappa\Upsilon_\eta(\mathcal A)}\ \mathcal{L}\tilde g+\hat g\right)^2(U-u)\right\}-1\right].
\end{multline}
Equations (\ref{m16}) and (\ref{m21}) then imply that
\begin{multline*}
\ell(u,\hat g)=\exp\left\{\imath\theta\ \exp\left[-\frac{1}{\hat g^2}
\left(\sqrt{\kappa\Upsilon_\eta(\mathcal A)}\mathcal{L}\tilde{g}+\hat{g}\right)^2(U-u)\right]
\right.\\
-
\mbox{$\frac{1}{4}$}\theta^2
\left(\frac{\kappa\Upsilon_\eta(\mathcal A)\mathcal{L}\tilde g_u}{\sqrt{\kappa\Upsilon_\eta(\mathcal A)}\mathcal{L}\tilde{g}+\hat{g}}\right)^2\times\\ 
\left.
\left[\exp\left\{-\frac{2}{\hat g^2}
\left(\sqrt{\kappa\Upsilon_\eta(\mathcal A)}\mathcal{L}\tilde{g}+\hat{g}\right)^2(U-u)\right\}-1\right] \right\}.
\end{multline*}
Taking $u=0$ yields
\begin{multline*}
\ell(0,\hat g)=
\Phi_{\hat g_U}(\theta)=
\exp\left\{\imath\theta\ \exp\left\{-\frac{1}{\hat g^2}
\left(\sqrt{\kappa\Upsilon_\eta(\mathcal A)}\mathcal{L}\tilde{g}+
\hat{g}\right)^2U\right\}\right. \\
-
\mbox{$\frac{1}{4}$}\theta^2
\left(\frac{\kappa\Upsilon_\eta(\mathcal A)\mathcal{L}\tilde g_u}{\sqrt{\kappa\Upsilon_\eta(\mathcal A)}\mathcal{L}\tilde{g}+\hat{g}}\right)^2
\times \\
\left.
\left[\exp\left\{-\frac{2}{\hat g^2}
\left(\sqrt{\kappa\Upsilon_\eta(\mathcal A)}\mathcal{L}\tilde{g}+
\hat{g}\right)^2U\right\}-1\right]\right\}.
\end{multline*}
\end{proof}
As the characteristic function in Lemma \ref{l2} is in the complex plane, we need to multiply by its conjugate and then take the square root to obtain $\bm{\sigma}_2^*$ which is in $\mathbb{S}^{I\times U}$.

\subsection{Match without interruption}
Following \cite{duckworth1998} we know, in one-day match each over can be represented as a multiple of $1/6$ which makes $u$ a continuous variable. The objective function is,
\begin{multline}\label{2}
\max_{\{W_i\in W\}}\overline{\mathbf{Z}}_T(\mathbf{W},u)= \\
\max_{\{W_i\in W\}}\E \int_0^{U}
\sum_{i=1}^{I}\sum_{m=1}^M\exp(-\rho_i m)\beta_iW_i(u)Z_{im}(u,w) du,
\end{multline}
where $U=50$. In Equation (\ref{2}), $\beta_i$ is the coefficient of the measure of the $i^{th}$ player's control. 

\begin{prop}\label{p0}
If team $T$'s objective is to maximize Equation (\ref{2}) subject to the run dynamics 
\begin{equation}\label{run}
d\mathbf{Z}(u,w)=\bm{\mu}[u,\mathbf{W}(u),\mathbf{Z}(u,w)]du+
\bm{\sigma}[u,\bm{\sigma}_2^*,\mathbf{W}(u),\mathbf{Z}(u,w)]d\mathbf{B}(u),
\end{equation} 
with Assumptions \ref{as0} and \ref{as1}, then under a continuous over system  of one-day cricket, player $i$'s coefficient is found by solving  
\begin{multline*}
\sum_{i=1}^{I}\sum_{m=1}^M\exp(-\rho_i m)\beta_i\ Z_{im}(u,w)+
\frac{\partial g[u,\mathbf{Z}(u,w)]}{\partial \mathbf{Z}}
\frac{\partial \bm{\mu}[u,\mathbf{W}(u),\mathbf{Z}(u,w)]}{\partial \mathbf{W}}
\frac{\partial \mathbf{W}}{\partial W_i } \\ 
\hspace{.5cm}+\mbox{$\frac{1}{2}$}
\sum_{i=1}^I\sum_{j=1}^I
\frac{\partial \bm\sigma^{ij}[u,\bm{\sigma}_2^*,\mathbf{W}(u),\mathbf{Z}(u,w)]}
{\partial \mathbf{W}}
\frac{\partial \mathbf{W}}{\partial W_i }
\frac{\partial^2 g[u,\mathbf{Z}(u,w)]}{\partial{Z_i\partial Z_j}}=0,
\end{multline*}
with respect to $\beta_i$, where the initial condition before the first bowl has been delivered is $\mathbf{0}_{I\times 1}$. Furthermore, if $\beta_i=\beta_j=\beta^*$, for all $i\neq j$, then
\begin{multline*}
\be^*(\mathbf{Z})=
-\left[\sum_{i=1}^{I}\sum_{m=1}^M\exp(-\rho_i m)Z_{im}(u,w)\right]^{-1}\times \\
\left[\frac{\partial g[u,\mathbf{Z}(u,w)]}{\partial{\mathbf{Z}}}
\frac{\partial \bm{\mu}[u,\mathbf{W}(u),\mathbf{Z}(u,w)]}{\partial \mathbf{W}} 
\frac{\partial \mathbf{W}}{\partial W_i}\right. \\ 
\left.
+\mbox{$\frac{1}{2}$}\sum_{i=1}^I\sum_{j=1}^I
\frac{\partial \bm\sigma^{ij}[u,\bm{\sigma}_2^*,\mathbf{W}(u),\mathbf{Z}(u,w)]}
{\partial \mathbf{W}}
\frac{\partial \mathbf{W}}{\partial W_i }
\frac{\partial^2 g[u,\mathbf{Z}(u,w)]}{\partial Z_i\partial Z_j}\right],
\end{multline*}
where $g[u,\mathbf{Z}(u,w)]\in C^2\left([0,50]\times \mathbb{R}^I\right)$ with $\mathbf{Y}(u)=g[u,\mathbf{Z}(u,w)]$ is a positive, non-decreasing penalization function vanishing at infinity which substitutes the run dynamics such that, $\mathbf{Y}(u)$ is an It\^o process. 
\end{prop}
\begin{proof}
Using Equations (\ref{2}) and (\ref{run}), with the zero initial condition, the Lagrangian of run dynamics over a $50$-over match is,
\begin{multline*}
\mathcal{L}_{0,U}(\mathbf{Z})=
\int_0^{U}\E_{u}\left\{\sum_{i=1}^{I}\sum_{m=1}^M
\exp(-\rho_i m)\beta_iW_i(u)Z_{im}(u,w)du\right. \\
+\lambda(u+du)[\mathbf{W}(u+du)-\mathbf{W}(u)-
\bm{\mu}[u,\mathbf{W}(u),\mathbf{Z}(u,w)]du\\
\left.\phantom{\int}
-\bm{\sigma}[u,\bm{\sigma}_2^*,\mathbf{W}(u),\mathbf{Z}(u,w)]d\mathbf{B}(u)]\right\},
\end{multline*}
where $\lambda$ is a non-negative Lagrange multiplier. Subdivide $[0,U]$ into $n$ equal over-intervals $[u,u+\epsilon]$. For any positive $\epsilon$ and normalizing constant $N_u>0$, define the run transition function as 
\begin{equation}\label{3.0}
\Psi_{u,u+\epsilon}(\mathbf{Z})=
\frac{1}{N_u}\int_{\mathbb{R}^I}\exp
\left[-\epsilon\ \mathcal{L}_{u,u+\epsilon}(\mathbf{Z})\right]
\Psi_u(\mathbf{Z})d\mathbf{Z},
\end{equation}
where $\Psi_u(\mathbf{Z})$ is the run transition function at the beginning of $u$ and $N_u^{-1}\ d\mathbf{Z}$ is a finite Riemann measure such that,
\begin{equation}\label{3.2}
\Psi_{0,U}(\mathbf{Z})=\frac{1}{N_u^n}
\int_{\mathbb{R}^{I\times n}}\exp\left[
-\epsilon\sum_{k=1}^n\mathcal{L}_{u,u+\epsilon}^k(\mathbf{Z})\right]
\Psi_0(\mathbf{Z})\prod_{k=1}^nd\mathbf{Z}^k,
\end{equation}
with the finite measure $N_u^{-n}\prod_{k=1}^nd\mathbf{Z}^k$ 
and initial transition function $\Psi_0(\mathbf{Z})>0$ for all $n\in\mathbb{N}$ \citep{fujiwara2017}.

Define $\Delta \mathbf{W}(\nu)=\mathbf{W}(\nu+d\nu)-\mathbf{W}(\nu)$, then Fubuni's theorem implies,
\begin{multline}
\mathcal{L}_{u,\tau}(\mathbf{Z})=\E_{u}\ \int_0^{U}\left\{\sum_{i=1}^{I}
\sum_{m=1}^M\exp(-\rho_i m)\beta_iW_i(\nu)Z_{im}(\nu,w)d\nu\right. \\
\left.
\phantom{\int}
+\lambda[\Delta\mathbf{W}(\nu)-
\bm{\mu}[\nu,\mathbf{W}(\nu),\mathbf{Z}(\nu,w)]d\nu-
\bm{\sigma}[\nu,\bm{\sigma}_2^*,\mathbf{W}(\nu),\mathbf{Z}(\nu,w)]d\mathbf{B}(\nu)]
\right\},
\end{multline}
where $\tau=u+\epsilon$. As we assume the run dynamics has drift and diffusion parts, $\mathbf{Z} (\nu,w)$ is an It\^o process, and $\mathbf{W}$ is a Markov control measure of
valuation of players, there exists a smooth function $g[\nu,\mathbf{Z}(\nu,w)]\in C_0^2([0,50]\times \mathbb{R}^I)$ such that $\mathbf{Y}(\nu)=g[\nu,\mathbf{Z}(\nu,w)]$ where $\mathbf{Y}(\nu)$ is an It\^o process  \citep{oksendal2003}. Assuming 
\begin{multline*}
g[\nu+\Delta \nu,\mathbf{Z}(\nu,w)+\Delta \mathbf{Z}(\nu,w)]=\\ 
\lambda[\Delta\mathbf{W}(\nu)-\mu[\nu,\mathbf{W}(\nu),\mathbf{Z}(\nu,w)]d\nu-
\sigma[\nu,\bm{\sigma}_2^*,\mathbf{W}(\nu),\mathbf{Z}(\nu,w)] d\mathbf{B}(\nu)],
\end{multline*}
for a very small interval around $u$ with $\epsilon\downarrow 0$, generalized It\^o's Lemma yields,
\begin{eqnarray*}
\epsilon\mathcal{L}_{u,\tau}(\mathbf{Z}) & = &  
\E_{u}\left\{\sum_{i=1}^{I}\sum_{m=1}^M
\epsilon\exp(-\rho_i m)\beta_iW_i(u)Z_{im}(u,w)+\epsilon g[u,\mathbf{Z}(u,w)] \right.\\
  & & +\epsilon g_u[u,\mathbf{Z}(u,w)]+\epsilon 
  g_{\mathbf{Z}}[u,\mathbf{Z}(u,w)]\bm\mu[u,\mathbf{W}(u),\mathbf{Z}(u,w)] \\
 & & 
+\epsilon g_{\mathbf{Z}}[u,\mathbf{Z}(u,w)]
 \bm{\sigma}[u,\bm{\sigma}_2^*,\mathbf{W}(u),\mathbf{Z}(u,w)]
 \Delta\mathbf{B}(u)\\ 
 & & \left.+\mbox{$\frac{1}{2}$} 
\sum_{i=1}^I\sum_{j=1}^I\epsilon
\bm{\sigma}^{ij}[u,\bm{\sigma}_2^*,\mathbf{W}(u),\mathbf{Z}(u,w)]
g_{Z_iZ_j}[u,\mathbf{Z}(u,w)]+o(\epsilon)\right\},
\end{eqnarray*}
where $\bm\sigma^{ij}[u,\bm{\sigma}_2^*,\mathbf{W}(u),\mathbf{Z}(u,w)]$ represents $\{i,j\}^{th}$ component of the variance-covarience matrix, $g_u=\partial g/\partial u$, $g_{\mathbf{Z}}=\partial g/\partial \mathbf{Z}$ and $g_{Z_iZ_j}=\partial^2 g/(\partial Z_i\ \partial Z_j)$, $\Delta B_i\ \Delta B_j=\delta^{ij}\ \epsilon$, $\Delta B_i\ \epsilon=\epsilon\ \Delta B_i=0$, and $\Delta Z_i(u)\ \Delta Z_j(u)=\epsilon$, where $\delta^{ij}$ is the Kronecker delta function. As $\E_u[\Delta \mathbf{B}(u)]=0$ and $\E_u[o(\epsilon)]/\epsilon\ra 0$, for $\epsilon\downarrow 0$, with the vector of initial conditions $\mathbf{0}_{I\times 1}$ dividing throughout by  $\epsilon$ and taking the conditional expectation we get,
\begin{eqnarray*}
\mathcal{L}_{u,\tau}(\mathbf{Z}) & = &
\sum_{i=1}^{I}\sum_{m=1}^M\exp(-\rho_i m)\beta_iW_i(u)Z_{im}(u,w)
+g[u,\mathbf{Z}(u,w)] \\ 
 & & +g_u[u,\mathbf{Z}(u,w)]
+ g_{\mathbf{Z}}[u,\mathbf{Z}(u,w)]\bm{\mu}[u,\mathbf{W}(u),\mathbf{Z}(u,w)] \\
 & & +\mbox{$\frac{1}{2}$}\sum_{i=1}^I\sum_{j=1}^I
 \bm{\sigma}^{ij}[u,\bm{\sigma}_2^*,\mathbf{W}(u),\mathbf{Z}(u,w)]
 g_{Z_iZ_j}[u,\mathbf{Z}(u,w)]+o(1).
\end{eqnarray*}
Suppose, there exists a vector $\mathbf{\xi}_{I\times 1}$  such that $\mathbf{Z}(u,w)_{I\times 1}=\mathbf{Z}(\tau,w)_{I\times 1}+\xi_{I\times 1}$. For a number $0<\eta<\infty$ assume $|\xi|\leq\eta\epsilon [\mathbf{Z}^T(u,w)]^{-1}$, which makes $\xi$ a very small number for each $\epsilon\downarrow 0$ and 
after defining a $C^2$ function 
\begin{eqnarray*}
f[u,\mathbf{W}(u),\xi] & = & 
\sum_{i=1}^{I}\sum_{m=1}^M\exp(-\rho_i m)\beta_iW_i(u)
[Z_{im}(\tau,w)+\xi]\\ 
 & & +g[u,\mathbf{Z}(\tau,w)+\xi]+g_u[u,\mathbf{Z}(\tau,w)+\xi] \\
 & & +g_{\mathbf{Z}}[u,\mathbf{Z}(\tau,w)+\xi]
 \bm{\mu}[u,\mathbf{W}(u),\mathbf{Z}(\tau,w)+\xi]\\
  & & +\mbox{$\frac{1}{2}$}\sum_{i=1}^I\sum_{j=1}^I
\bm{\sigma}^{ij}[u,\bm{\sigma}_2^*,\mathbf{W}(u),\mathbf{Z}(\tau,w)+\xi]\times \\
 & & g_{Z_iZ_j}[u,\mathbf{Z}(\tau,w)+\xi],
\end{eqnarray*} 
we have,
\begin{multline}\label{8}
\Psi_u^\tau(\mathbf{Z})+
\epsilon\frac{\partial \Psi_u^\tau(\mathbf{Z})}{\partial u} =
\frac{1}{N_u}\Psi_u^\tau(\mathbf{Z}) 
\int_{\mathbb{R}^{I}}\exp\left\{-\epsilon f[u,\mathbf{W}(u),\xi]\right\}d\xi\\
+\frac{1}{N_u}\frac{\partial \Psi_u^\tau(\mathbf{Z})}{\partial \mathbf{Z}}
 \int_{\mathbb{R}^I} 
\xi\exp\left\{-\epsilon f[u,\mathbf{W}(u),\xi]\right\} d\xi
 +o(\epsilon^{1/2}).
\end{multline}
For $\epsilon\downarrow0$, $\Delta \mathbf{Z}\downarrow0$ and,
\begin{multline*}
f[u,\mathbf{W}(u),\xi]=f[u,\mathbf{W}(u),\mathbf{Z}(\tau,w)]+
\sum_{i=1}^{I}f_{Z_i}[u,\mathbf{W}(u),\mathbf{Z}(\tau,w)][\xi_i-Z_i(\tau,w)]\\
+\mbox{$\frac{1}{2}$}\sum_{i=1}^{I}\sum_{j=1}^{I}
f_{Z_iZ_j}[u,\mathbf{W}(u),\mathbf{Z}(\tau,w)]
[\xi_i-Z_i(\tau,w)][\xi_j-Z_j(\tau,w)]+o(\epsilon),
\end{multline*}
assume there exists a symmetric, positive definite and non-singular Hessian matrix $\mathbf{\Theta}_{I\times I}$ and a vector $\mathbf{R}_{I\times 1}$ such that,
\begin{multline*}
\Psi_u^\tau(\mathbf{Z})+\epsilon\frac{\partial \Psi_u^\tau(\mathbf{Z})}{\partial u}
=\frac{1}{N_u}\sqrt{\frac{(2\pi)^{I}}{\epsilon |\mathbf{\Theta}|}}
\exp\{-\epsilon f[u,\mathbf{W}(u),\mathbf{Z}(\tau,w)]+
\mbox{$\frac{1}{2}$}\epsilon\mathbf{R}^T\mathbf{\Theta}^{-1}\mathbf{R}\}\\
\times
\left\{\Psi_u^\tau(\mathbf{Z})+[\mathbf{Z}(\tau,w)+\mbox{$\frac{1}{2}$}
(\mathbf{\Theta}^{-1}\mathbf{R})]
\frac{\partial \Psi_{\mathbf{u}}^\tau(\mathbf{Z})}{\partial \mathbf{Z}}\right\}
+o(\epsilon^{1/2}).
\end{multline*}
Assuming $N_u=\sqrt{(2\pi)^{I}/(\epsilon |\mathbf{\Theta}|)}>0$, we get Wick rotated 
Schr\"odinger type equation as,
\begin{multline}\label{13}
\Psi_u^\tau(\mathbf{Z})+\epsilon\frac{\partial \Psi_u^\tau(\mathbf{Z})}{\partial u}= 
\{1-\epsilon f[u,\mathbf{W}(u),\mathbf{Z}(\tau,w)]+
\mbox{$\frac{1}{2}$}\epsilon\mathbf{R}^T\mathbf{\Theta}^{-1}\mathbf{R}\}\times \\
\left\{\Psi_u^\tau(\mathbf{Z})+[\mathbf{Z}(\tau,w)+\mbox{$\frac{1}{2}$} 
(\mathbf{\Theta}^{-1}\mathbf{R})]
\frac{\partial \Psi_{\mathbf{u}}^\tau(\mathbf{Z})}{\partial \mathbf{Z}}\right\}
+o(\epsilon^{1/2}).
\end{multline}
For any finite positive number $\eta$ we know $\mathbf{Z}(\tau,w)\leq\eta\epsilon|\xi^T|^{-1}$. Then there exists $|\mathbf{\Theta}^{-1}\mathbf{R}|\leq 2 \eta\epsilon|1-\xi^T|^{-1}$ such that for $\epsilon\downarrow 0$ we have, $\big|\mathbf{Z}(\tau,w)+\mbox{$\frac{1}{2}$}\ \left(\mathbf{\Theta}^{-1}\ \mathbf{R}\right)\big|\leq\eta\epsilon$, for $|\mathbf{\Theta}^{-1}\mathbf{R}|\leq 2 \eta\epsilon|1-\xi^T|^{-1}$, where $\xi^T$ is the transposition of $\xi$ and differentiating Equation (\ref{13}) with respect to $W_i$  yields,
\begin{equation}\label{16}
-\frac{\partial}{\partial W_i}
f[u,\mathbf{W}(u),\mathbf{Z}(\tau,w)]\Psi_u^\tau(\mathbf{Z})=0.
\end{equation}
In Equation (\ref{16}) as $\Psi_u^\tau(\mathbf{Z})$ is a transition function $\Psi_u^\tau(\mathbf{Z})\neq0$.  Hence,\\ $\frac{\partial }{\partial W_i}f[u,\mathbf{W}(u),\mathbf{Z}(\tau,w)]=0$. We know, $\mathbf{Z}(\tau,w)=\mathbf{Z}(u,w)-\xi$ and for $\xi\ra 0$ as we are looking for some stable solution therefore, in Equation (\ref{16}) $\mathbf{Z}(\tau,w)$ can be replaced by $\mathbf{Z}(u,w)$. Therefore,
\begin{multline}\label{16.1}
f[u,\mathbf{W}(u),\mathbf{Z}(u,w)]=
\sum_{i=1}^{I}\sum_{m=1}^M\exp(-\rho_i m)\beta_iW_i(u)Z_{im}(u,w)\\
+g[u,\mathbf{Z}(u,w)]+ g_u[u,\mathbf{Z}(u,w)]+
g_{\mathbf{Z}}[u,\mathbf{Z}(u,w)]\bm{\mu}[u,\mathbf{W}(u),\mathbf{Z}(u,w)]\\
+\mbox{$\frac{1}{2}$}\sum_{i=1}^I\sum_{j=1}^I
\bm{\sigma}^{ij}[u,\bm{\sigma}_2^*,\mathbf{W}(u),\mathbf{Z}(u,w)]
g_{Z_iZ_j}[u,\mathbf{Z}(u,w)].
\end{multline}
Equations (\ref{16}) and (\ref{16.1}) imply
\begin{multline}\label{16.2}
\sum_{i=1}^{I}\sum_{m=1}^M\exp(-\rho_i m)\beta_iZ_{im}(u,w)\\
+g_{\mathbf{Z}}[u,\mathbf{Z}(u,w)]
\frac{\partial \bm{\mu}[u,\mathbf{W}(u),\mathbf{Z}(u,w)]}{\partial \mathbf{W}}
\frac{\partial \mathbf{W}}{\partial W_i }\\ 
+\mbox{$\frac{1}{2}$}\sum_{i=1}^I\sum_{j=1}^I
g_{Z_iZ_j}[u,\mathbf{Z}(u,w)]
\frac{\partial \bm{\sigma}^{ij}[u,\bm{\sigma}_2^*,\mathbf{W}(u),\mathbf{Z}(u,w)]}
{\partial \mathbf{W}}
\frac{\partial \mathbf{W}}{\partial W_i }=0.
\end{multline}
Assume $\beta_i=\beta_j=\beta^*$ for all $i\neq j$ then,
\begin{multline*}
\be^*(\mathbf{Z})=
-\left[\sum_{i=1}^{I}\ \sum_{m=1}^M\ \exp(-\rho_i m)Z_{im}(u,w)\right]^{-1}\times\\
\left[g_{\mathbf{Z}}[u,\mathbf{Z}(u,w)]
\frac{\partial \bm{\mu}[u,\mathbf{W}(u),\mathbf{Z}(u,w)]}{\partial \mathbf{W}}
\frac{\partial \mathbf{W}}{\partial W_i }
\right. \\ 
\left.+\mbox{$\frac{1}{2}$}\sum_{i=1}^I\sum_{j=1}^I
g_{Z_iZ_j}[u,\mathbf{Z}(u,w)]
\frac{\partial \bm{\sigma}^{ij}[u,\bm{\sigma}_2^*,\mathbf{W}(u),\mathbf{Z}(u,w)]}
{\partial \mathbf{W}}
\frac{\partial \mathbf{W}}{\partial W_i }\right].
\end{multline*} 
\end{proof}

\subsection{Match with interruption by rain} 
Suppose, for a $[0,U]$ over one-day match the game stops after $\widetilde U-1$ overs because of the rain. After that there are two possibilities: first, if the rain is heavy, the game will not resume; secondly, if the rain is not heavy and stops after certain point of time then, after  getting water out of the field, the match might be resumed. Based on the severity of the rain and the equipment used to get the water out from the field the match resumes for $(\widetilde U,U-\varepsilon]$ overs where $\varepsilon\geq 0$. 

\begin{definition}\label{de3}
For a probability space $(\Omega, \mathcal{F}_u^{\mathbf{Z}},\mathcal{P})$ with sample space $\Omega$, filtration at $u^{th}$ over of run ${\mathbf{Z}}$ as $\{\mathcal{F}_u^{\mathbf{Z}}\}\subset\mathcal{F}_u$, a probability measure $\mathcal{P}$ and a Brownian motion for rain $\mathbf{B}_u$ with the form $\mathbf{B}_u^{-1}(E)$ such that for $u\in[\widetilde U,U-\varepsilon]$, $E\subseteq\mathbb{R}$ is a Borel set. If $\widetilde{U}$ is the game stopping over and ${b}\in \mathbb{R}$ is a rain measure then $\widetilde U:=\inf\{u\geq 0|\ \mathbf{B}_u>{b}\}$. 
\end{definition} 

\begin{defn}\label{de4}
Let $\delta_{\mathfrak u}:[\widetilde U,U-\varepsilon]	\ra(0,\infty)$ be a $C^2(u\in[\widetilde U,U-\varepsilon])$ over-process of a one-day match such that, it replaces stochastic process by It\^o's Lemma. Then $\delta_{\mathfrak u}$ is a stochastic gauge of that match if  $\widetilde U:=\mathfrak u+\delta_{\mathfrak u}$ is a stopping over for each $\mathfrak u\in[\widetilde U,U-\varepsilon]$ and $\mathbf B_{\mathfrak u}>b$, where $\mathfrak{u}$ is the new over after resampling the stochastic interval $[\widetilde U,U-\varepsilon]$.
\end{defn}

\begin{defn}\label{de5}
Given a stochastic over interval $\hat I=[\widetilde U,U-\varepsilon]\subset\mathbb{R}$, a stochastic tagged partition of a one-day match is a finite set of ordered pairs $\mathcal{D}=\{(\mathfrak u_i,\hat I_i):\ i=1,2,...,p\}$ such that $\hat I_i=[x_{i-1},x_i]\subset[\widetilde U,U-\varepsilon]$, $\mathfrak u_i\in \hat I_i$, $\cup_{i=1}^p \hat I_i=[\widetilde U,U-\varepsilon]$ and for $i\neq j$ we have $\hat I_i\cap \hat I_j=\{\emptyset\}$. The point $\mathfrak u_i$ is the tag partition of the stochastic over-interval $\hat I_i$.
\end{defn}

\begin{defn}\label{de6}
If $\mathcal{D}=\{(\mathfrak u_i,\hat I_i): i=1,2,...,p\}$ is a tagged partition of stochastic over-interval $\hat I$ and $\delta_{\mathfrak u}$ is a stochastic gauge on $\hat I$, then $\mathcal D$ is a stochastic $\delta$-fine if $\hat I_i\subset \delta_{\mathfrak u}(\mathfrak u_i)$ for all $i=1,2,...,p$, where $\delta(\mathfrak u)=(\mathfrak u-\delta_{\mathfrak u}(\mathfrak u),\mathfrak u+\delta_{\mathfrak u}(\mathfrak u))$.
\end{defn}

For a tagged partition defined in Definitions \ref{de5} and \ref{de6}, and a function $\tilde f:[\widetilde U,U-\varepsilon]\times \mathbb{R}^{2I\times \widehat U}\times\Omega\ra\mathbb{R}^{I\times \hat U}$ the Riemann sum of $\mathcal D$ is defined as 
$
S(\tilde f,\mathcal D)=(\mathcal D_\delta)\sum \tilde 
f(\mathfrak u,\hat I,\mathbf W,\mathbf Z)=
\sum_{i=1}^p \tilde f(\mathfrak u_i,\hat I_i,\mathbf W,\mathbf Z),
$
where $\mathcal{D}_\delta$ is a $\delta$-fine division of 
$\mathbb{R}^{I\times \widehat U}$ 
with point-cell function 
$\tilde f(\mathfrak u_i,\hat I_i,\mathbf W,\mathbf Z)=
\tilde f(\mathfrak u_i,\mathbf W,\mathbf Z)\ell(\hat I_i)$, 
where $\ell$ is the length of the over interval and 
$\hat{U}=(U-\varepsilon)-\tilde{U}$ \citep{kurtz2004}.

\begin{defn}\label{de7}
An over integrable function $\tilde f(\mathfrak u,\hat I,\mathbf W,\mathbf Z)$ on $\mathbb{R}^{I\times \widehat U}$, with integral 
$
\mathbf a=\int_{\widetilde U}^{U-\varepsilon}\tilde{f}(\mathfrak u, \hat I,\mathbf W,\mathbf Z)
$
is stochastic Henstock-Kurzweil type integrable on $\hat I$ if, for a given vector $\hat{\bm\epsilon}>0$, there exists a stochastic $\delta$-gauge in $[\widetilde U,U-\varepsilon]$ such that for each stochastic $\delta$-fine partition $\mathcal D_\delta$ in $\mathbb R^{I\times \widehat U}$ we have, \\
$
\E_{\mathfrak u}\left\{\left|\mathbf a-(\mathcal D_\delta)\sum \tilde f(\mathfrak u,\hat I,\mathbf W,\mathbf Z)\right|\right\}<\hat{\bm\epsilon},
$
where $\E_{\mathfrak u}$ is the conditional expectation on run $\mathbf{Z}$ at sample over $\mathfrak u\in[\widetilde U,U-\varepsilon]$ of a non-negative function $\tilde f$ after the rain stops.
\end{defn}

\begin{prop}\label{p1}
Define 
$
\mathfrak h=\exp\left\{-
\tilde\epsilon\E_{\mathfrak u}\left[\int_{\mathfrak u}^{\mathfrak u+
\tilde\epsilon}\tilde f(\mathfrak u,\hat I,\mathbf W,\mathbf Z)\right]\right\}
\Psi_{\mathfrak u}(\mathbf{Z})d\mathbf Z.
$
If for a small sample over interval $[\mathfrak u,\mathfrak u+\tilde\epsilon]$, 
$
\frac{1}{N_{\mathfrak u}}\int_{\mathbb R^{2I\times\widehat U\times I}}\mathfrak h
$
exists for a conditional gauge $\gamma=[\delta,\omega(\delta)]$, then the indefinite integral of $\mathfrak h$, 
$
\mathbf H(\mathbb R^{2I\times\widehat U\times I})=\frac{1}{N_{\mathfrak u}}\int_{\mathbb R^{2I\times\widehat U\times I}}\mathfrak h
$
exists as Stieltjes function in $ \mathbf E([\mathfrak u,\mathfrak u+\tilde\epsilon]
\times \mathbb R^{2I\times\widehat U}\times \Omega\times\mathbb R^I)$ for all $N_{\mathfrak u}>0$.
\end{prop}
\begin{proof}
Define a gauge $\gamma=[\delta,\omega(\delta)]$ for all possible combinations of a $\delta$ gauge in $[\widetilde U,U-\varepsilon]\times \mathbb R^{2I\times\widehat U}\times \Omega$ and $\omega(\delta)$-gauge in $\mathbb R^{2I\times\widehat U\times I}$ such that it is a cell in $[\widetilde U,U-\varepsilon]\times \mathbb R^{2I\times\widehat U}\times \Omega\times\mathbb R^{2I\times\widehat U\times I}$, where $\omega(\delta):\mathbb{R}^{2I\times\widehat U\times I}\ra(0,\infty)^{2I\times\widehat U\times I}$ is at least a $C^1$ function. The reason behind considering $\omega(\delta)$ as a function of $\delta$ is because, after rain stops, if the $u^{th}$ over proceeds then we can get a corresponding sample over $\mathfrak{u}$ and the batsman has the opportunity to score run. Let $\mathcal D_\gamma$ be a stochastic $\gamma$-fine in cell $\mathbf E$ in $[\widetilde U,U-\varepsilon]\times \mathbb R^{2I\times\widehat U}\times \Omega\times\mathbb R^I$. For any $\bm{\varepsilon}>0$ and for a $\delta$-gauge in $[\widetilde U,U-\varepsilon]\times \mathbb R^{2I\times\widehat U}\times \Omega$ and $\omega(\delta)$-gauge in $\mathbb R^{2I\times\widehat U\times I}$ choose a $\gamma$ so that 
$
|\frac{1}{N_{\mathfrak u}}(\mathcal D_\gamma)\sum \mathfrak h-\mathbf H(\mathbb R^{2I\times\widehat U\times I})|<\mbox{$\frac{1}{2}$}|\mathfrak u-\mathfrak u'|,
$
where $\mathfrak u'=\mathfrak u+\tilde{\epsilon}$. Assume two disjoint sets $E^a$ and $E^b=[\mathfrak u,\mathfrak u+\tilde{\epsilon}]\times \mathbb R^{2I\times\widehat U}\times \Omega\times\mathbb \{R^I\setminus E^a\}$ such that $E^a\cup E^b=E$ . As the domain of $\tilde f$ is a $2$-sphere, Theorem $3$ in \cite{muldowney2012} implies there is a gauge $\gamma_a$ for set $E^a$ and a gauge $\gamma_b$ for set $E^b$ with $\gamma_a\prec\gamma$ and $\gamma_b\prec\gamma$, so that both the gauges conform in their respective sets. For every $\delta$-fine in $[\mathfrak u,\mathfrak u']\times\mathbb{R}^{2I\times\widehat U}\times\Omega$ and a positive 
$\tilde{\epsilon}=|\mathfrak u-\mathfrak u'|$, 
if a $\gamma_a$-fine division $\mathcal D_{\gamma_a}$ is of the set $E^a$ and $\gamma_b$-fine division $\mathcal D_{\gamma_b}$ is of the set $E^b$, then by the restriction axiom we know that $\mathcal D_{\gamma_a}\cup\mathcal D_{\gamma_b}$ is a $\gamma$-fine division of $E$. Furthermore, as $E^a\cap E^b=\emptyset$
\begin{align}
\mbox{$\frac{1}{N_{\mathfrak u}}$}\left(\mathcal D_{\gamma_a}\cup\mathcal D_{\gamma_b}\right)\sum\mathfrak h=\mbox{$\frac{1}{N_{\mathfrak u}}$}\left[(\mathcal D_{\gamma_a})\sum \mathfrak h+(\mathcal D_{\gamma_b})\sum\mathfrak h\right]=\a+\be.\notag
\end{align}
Let us assume that for every $\delta$-fine we can subdivide the set $E^b$ into two disjoint subsets $E_1^b$ and $E_2^b$ with their $\gamma_b$-fine divisions given by $\mathcal D_{\gamma_b}^1$ and $\mathcal D_{\gamma_b}^2$, respectively. Therefore, their Riemann sum can be written as $\be_1=\frac{1}{N_{\mathfrak u}}(\mathcal D_{\gamma_b}^1)\sum\mathfrak h$ and $\be_2=\frac{1}{N_{\mathfrak u}}(\mathcal D_{\gamma_b}^2)\sum\mathfrak h$, respectively. Hence, for a small sample over interval $[\mathfrak u,\mathfrak u']$,
$
\big|\a+\be_1-\mathbf H(\mathbb R^{2I\times\widehat U\times I})\big|\leq\mbox{$\frac{1}{2}$}|\mathfrak u-\mathfrak u'|
$
and
$
\big|\a+\be_2-\mathbf H(\mathbb R^{2I\times\widehat U\times I})\big|\leq\mbox{$\frac{1}{2}$}|\mathfrak u-\mathfrak u'|.
$
Therefore,
\begin{eqnarray}\label{Cauchy}
|\be_1-\be_2| & = & 
\left|\left[\a+\be_1-\mathbf H(\mathbb R^{2I\times\widehat U\times I})\right]-
\left[\a+\be_2-\mathbf H(\mathbb R^{2I\times\widehat U\times I})\right]\right|\notag \\
 & \leq & \left|\a+\be_1-\mathbf H(\mathbb R^{2I\times\widehat U\times I})\right|+
 \left|\a+\be_2-\mathbf H(\mathbb R^{2I\times\widehat U\times I})\right| \notag \\
  & \leq & |\mathfrak u-\mathfrak u'|.
\end{eqnarray}
Equation (\ref{Cauchy}) implies that the Cauchy integrability of $\mathfrak h$ is satisfied, and \\
$
\mathbf H(\mathbb R^{2I\times\widehat U\times I})=
\frac{1}{N_{\mathfrak u}}\int_{\mathbb R^{2I\times\widehat U\times I}}\mathfrak h.
$
Now consider two disjoint set $M^1$ and $M^2$ in $\mathbb R^{2I\times\widehat U\times I}$ such that $M=M^1\cup M^2$ with their corresponding integrals $\mathbf H(M^1)$, $\mathbf H(M^2)$, and $\mathbf H(M)$. Suppose $\gamma$-fine divisions of $M^1$ and $M^2$ are given by $\mathcal D_{\gamma_1}$ and $\mathcal D_{\gamma_2}$, respectively, with their Riemann sums for $\mathfrak h$ are $m_1$ and $m_2$. Equation (\ref{Cauchy}) implies, $\big|m_1-\mathbf H(M^1)\big|\leq\big|\mathfrak u-\mathfrak u'\big|$ and $\big|m_2-\mathbf H(M^2)\big|\leq\big|\mathfrak u-\mathfrak u'\big|$. Hence, $\mathcal D_{\gamma_1}\cup\mathcal D_{\gamma_2}$ is a $\gamma$-fine division of $M$. Let $m=m_1+m_2$ then Equation (\ref{Cauchy}) implies $\big|m-\mathbf H(M)\big|\leq |\mathfrak u-\mathfrak u'|$ and
\begin{eqnarray*}
|[\mathbf H(M^1)+\mathbf H(M^2)]-\mathbf H(M)| & \leq &
|m-\mathbf{H}(M)|+|m_1-\mathbf H(M^1)|+ \\
 & & |m_2-\mathbf{H}(M^2)| \\
 & \leq & 3|\mathfrak u-\mathfrak u'|.
\end{eqnarray*}
Therefore, $\mathbf H(M)=\mathbf H(M^1)+\mathbf H(M^2)$ and it is Stieljes.
\end{proof}

\begin{coro}\label{c0}
If $\mathfrak h$ is integrable on $\mathbb R^{2I\times\widehat U\times I}$ as in Proposition \ref{p1}, then for a given small continuous sample over the interval $[\mathfrak u,\mathfrak u']$ with 
$\tilde{\epsilon}=\mathfrak u'-\mathfrak u>0$, 
there exists a $\gamma$-fine division $\mathcal D_\gamma$ in $\mathbb R^{2I\times\widehat U\times I}$ such that,
\[
|(\mathcal D_\gamma)\mathfrak h[\mathfrak u,\hat I,\hat I(\mathbf Z),\mathbf W,\mathbf Z]-\mathbf H(\mathbb R^{2I\times\widehat U\times I})|\leq\mbox{$\frac{1}{2}$}|\mathfrak u-\mathfrak u'|<\tilde{\epsilon},
\]
where $\hat I(\mathbf Z)$ is the interval of run $\mathbf Z$ in $\mathbb R^{2I\times\widehat U\times I}$. This integral is a stochastic It\^o-Henstock-Kurtzweil-McShane-Feynman-Liouville type path integral in run dynamics of a sample over after the beginning of an one-day match after rain interruption.
\end{coro}

The objective function after rain is,
\begin{multline}\label{r}
\max_{\{W_i\in W\}} \widehat{\mathbf{Z}}_T'(\mathbf{W},\mathfrak{u})= \\
\max_{\{W_i\in W\}}\E\int_{\widetilde U}^{U-\varepsilon} \sum_{i=1}^{I} \sum_{m=1}^M 
\exp(-\rho_i m) \beta_i W_i(\mathfrak{u}) Z_{im}(\mathfrak{u},w) d\mathfrak{u}.
\end{multline} 

Consider a domain $\mathbb{D}$ in a $2$-sphere $\mathbb{S}$. First, we know that after the rain the pacers will have extra swing, measured by the difference of the average swing before rain to after rain say, $\varphi_1\in\mathbb{D}$. For $p$-partitions of $[\widetilde U,U-\varepsilon]$, define $\varphi_1=\sum_{i=1}^p \varphi_{1i} \a_i$, where $\varphi_i$ is the difference at $\mathfrak u_i$-th over and $\a_i$ is the orthonormal basis. Second, the measure the slowness of outfield is $\varphi_2=\sum_{i=1}^p \varphi_{2i} \tilde{\be}_i$ 
where $\varphi_{2i}$ is the difference in the speed of a ball after a batsman offers a shot and 
$\tilde{\be}_i$ is an orthonormal basis. 
As $\varphi_{1i}$ and $\varphi_{2i}$ vary in each $\mathfrak u_i$, we assume they are random variables with the Dirichlet inner products $(\varphi_{11},\varphi_{12})_{\nabla}:=(2\pi)^{-1} \int_{\mathbb D}\nabla \varphi_{11}(\mathfrak u).\nabla \varphi_{12}\ d\mathfrak u$  and $(\varphi_{21},\varphi_{22})_{\nabla}:=(2\pi)^{-1} \int_{\mathbb D}\nabla \varphi_{21}(\mathfrak u).\nabla \varphi_{22}\ d\mathfrak u$ such that $\varphi:=\varphi_1+\varphi_2$, where $\nabla$ is the gradient vector. Therefore, $\varphi$ is an instance of a centered Gaussian free field on a bounded simply connected domain $\mathbb{D}$ with zero boundary condition \citep{duplantier2011}. The two pairs $(\mathbb D,\varphi)$ and $(\widehat{\mathbb D},\hat\varphi)$ are equivalent in an $\sqrt{8/3}$-Liouville quantum gravity if there exists a conformal map $\varpi:\widehat{\mathbb D}\ra\mathbb D$ such that, $\hat{\varphi}=\varphi\ o\  \varpi+Q\log|\varpi'|$, with $Q=\sqrt{3/2}+\sqrt{2/3}$ and $\gamma=\sqrt{8/3}$ \citep{sheffield2016}. The main importance of $\sqrt{8/3}$-Liouville quantum gravity surface is that, its natural measure is a limit of regularized versions of $\exp\left\{\sqrt{8/3}\ \varphi(\mathfrak u)\right\}d\mathfrak u$ with $d\mathfrak{u}$ being a stochastic Henstock type of measure in $\mathbb{D}$ \citep{gwynne2016}. After including this part in the run dynamics with $\lambda_{\mathfrak{u}}$ being the constant multiplier of $\sqrt{8/3}$-Liouville quantum gravity Equation (\ref{run}) becomes,
\begin{multline}\label{run0}
d\mathbf{Z}(\mathfrak u,w)=\bm{\mu}[\mathfrak u,\mathbf{W}(\mathfrak u),\mathbf{Z}
(\mathfrak u,w)]d\mathfrak u\\
+\exp[\sqrt{8/3}\varphi(\mathfrak u)]d\mathfrak u+
\hat{\bm{\sigma}}
[\mathfrak u,\bm{\sigma}_2^*,\mathbf{W}(\mathfrak u),\mathbf{Z}(\mathfrak u,w)]
d\mathbf{B}(\mathfrak u),
\end{multline}
and the Liouville like action function on the run dynamics after the match starts after the rain is
\begin{multline}\label{run1}
\mathcal{L}_{\tilde{U},U-\varepsilon}(\mathbf{Z})=
\int_{\tilde{U}}^{U-\varepsilon}\E_{\mathfrak{u}} \left\{\sum_{i=1}^{I}\sum_{m=1}^M
\exp(-\rho_i m)\beta_iW_i(\mathfrak u)Z_{im}(\mathfrak u,w)d\mathfrak u \right.\\
+\lambda_{\mathfrak{u}}[\mathbf{W}(\mathfrak{u}+d\mathfrak{u})]-
\mathbf{W}(\mathfrak{u})
-\bm{\mu}[\mathfrak{u},\mathbf{W}(\mathfrak{u}),\mathbf{Z}(\mathfrak{u},w)]
d\mathfrak{u}\\
\left.\phantom{\int}
-\exp[\sqrt{8/3}\varphi(\mathfrak u)]d\mathfrak{u}-
\hat{\bm{\sigma}}[\mathfrak u,\bm{\sigma}_2^*,\mathbf{W}(\mathfrak u),
\mathbf{Z}(\mathfrak u,w)]d\mathbf{B}(\mathfrak u)]\right\}.
\end{multline}
The stochastic part of the Equation (\ref{run0}) becomes $\hat{\bm{\sigma}}$ as $\hat{\lambda}_1>\lambda_1$. Equation (\ref{run1}) follows Definition \ref{de7} such that $\mathbf a=\mathcal{L}_{\widetilde U,U-\varepsilon}(\mathbf Z)$ and it is integrable according to Corollary \ref{c0}.

\begin{prop}\label{p2}
If team $T$'s objective is to maximize Equation (\ref{r}) subject to the run dynamics in Equation (\ref{run0}) such that, Assumptions \ref{as0}-\ref{as2} hold with Lemma \ref{l2}, Proposition \ref{p1} and Corollary \ref{c0}, then after a rain stoppage under a continuous sample over the system of the match, the coefficient is 
\begin{multline*}
\be^*(\mathbf{Z})=
-\left[\sum_{i=1}^{I}\sum_{m=1}^M\ \exp(-\rho_i m)Z_{im}(\mathfrak u,w)\right]^{-1}\times\\
\left[\frac{\partial g^a[\mathfrak u,\mathbf{Z}(\mathfrak u,w)]}{\partial{\mathbf{Z}}}
\frac{\partial
	\{\bm{\mu}[\mathfrak u,\mathbf{W}(\mathfrak u),\mathbf{Z}(\mathfrak u,w)] +
	\exp[\varphi(\mathfrak u)\sqrt{8/3}]\}}{\partial \mathbf{W}}
\frac{\partial \mathbf{W}}{\partial W_i } \right.\\ 
\left.
+\mbox{$\frac{1}{2}$}\sum_{i=1}^I\sum_{j=1}^I
\frac{\partial \hat{\bm\sigma}^{ij}[\mathfrak u,\bm{\sigma}_2^*,\mathbf{W}(\mathfrak u),
	\mathbf{Z}(\mathfrak u,w)]}{\partial \mathbf{W}}
\frac{\partial \mathbf{W}}{\partial W_i }
\frac{\partial^2 g^a[\mathfrak u,\mathbf{Z}(\mathfrak u,w)]}{\partial Z_i\partial Z_j}
\right],
\end{multline*} 
where $\beta_i=\beta_j=\beta^*$ for all $i\neq j$, $\mathbf Z_{\widetilde U}$ is the initial run condition, function\\ $g^a\left[\mathfrak u,\mathbf{Z}(\mathfrak u,w)\right]\in C_0^2\left([\widetilde U,U-\varepsilon]\times \mathbb R^{2I\times \widehat U}\times \mathbb{R}^I\right)$ with $\mathbf{Y}(\mathfrak u)=g^a\left[\mathfrak u,\mathbf{Z}(\mathfrak u,w)\right]$ is a positive, non-decreasing penalization function vanishing at infinity which substitutes for the run dynamics such that, $\mathbf{Y}(\mathfrak u)$ is an It\^o process.
\end{prop}

\begin{proof}
For a positive Lagrangian multiplier $\lambda_{\mathfrak u}$, with initial run condition $\mathbf Z_{\widetilde U}$ the run dynamics are expressed in Equation (\ref{run1}) such that such that Definition \ref{de7}, Proposition \ref{p1} and Corollary \ref{c0} hold. Subdivide $[\widetilde U,U-\varepsilon]$ into $n$ equally distanced small over-intervals $[\mathfrak u,\mathfrak u']$ such that $\tilde\epsilon\downarrow 0$, where $\mathfrak u'=\mathfrak u+\tilde\epsilon$.  For any positive $\tilde\epsilon$ and normalizing constant $N_{\mathfrak u}>0$, the run transition function is 
\[
\Psi_{\widetilde U,U-\tilde\varepsilon}(\mathbf{Z})=
\frac{1}{(N_{\mathfrak u})^n}
\int_{\mathbb{R}^{2I\times\widehat U\times I\times n}} \exp\bigg\{-\tilde\epsilon \sum_{k=1}^n\mathcal{L}_{\mathfrak u,\mathfrak u+\tilde\epsilon}^k(\mathbf{Z}) \bigg\}\Psi_0(\mathbf{Z}) \prod_{k=1}^n d\mathbf{Z}^k,
\]
with finite measure $\left(N_{\mathfrak u}\right)^{-n}\prod_{k=1}^{n}d\mathbf Z^k$ satisfying Corollary \ref{c0} with its initial run transition function after the rain stops as $\Psi_{\widetilde U}(\mathbf Z)>0$ for all $n\in\mathbb N$.  Define $\Delta \mathbf{W}(\nu)=\mathbf{W}(\nu+d\nu)-\mathbf{W}(\nu)$, then for $\tilde\epsilon\downarrow 0$ we have,
\begin{multline*}
\mathcal{L}_{\mathfrak u,\mathfrak u'}(\mathbf{Z})
=\E_{\mathfrak u} \int_{\widetilde U}^{U-\varepsilon} 
\left\{\sum_{i=1}^{I} \sum_{m=1}^M \exp(-\rho_i m) \beta_i W_i(\nu) Z_{im}(\nu,w)d\nu
\right.\\
+\lambda_\nu[\Delta \mathbf{W}(\nu)-
\{\bm{\mu}[\nu,\mathbf{W}(\nu),\mathbf{Z}(\nu,w)] +
\exp[\varphi(\nu)\sqrt{8/3}]\}d\nu\\
\left.\phantom{\int}
-\hat{\bm\sigma}[\nu,\bm{\sigma}_2^*,\mathbf{W}(\nu),\mathbf{Z}(\nu,w)]
d\mathbf{B}(\nu)]\right\}.
\end{multline*}
There exists a smooth function $g^a[\nu,\mathbf{Z}(\nu,w)]\in C^2\left([\widetilde U,U-\varepsilon]\times\mathbb R^{2I\times\widehat U}\times \mathbb{R}^I\right)$ such that $\mathbf{Y}(\nu)=g^a[\nu,\mathbf{Z}(\nu,w)]$ with $\mathbf{Y}(\nu)$ being an It\^o's process. Assume 
\begin{multline*}
g^a\left[\nu+\Delta\nu,\mathbf Z(\nu,w)+\Delta\mathbf Z(\nu,w)\right]=\\
\lambda_\nu[\Delta \mathbf{W}(\nu)-
\{\bm{\mu}[\nu,\mathbf{W}(\nu),\mathbf{Z}(\nu,w)] +
\exp[\varphi(\nu)\sqrt{8/3}]\}d\nu\\
-\hat{\bm\sigma}[\nu,\bm{\sigma}_2^*,\mathbf{W}(\nu),\mathbf{Z}(\nu,w)]
d\mathbf{B}(\nu)].
\end{multline*}
For a very small sample over-interval around $\mathfrak u$ with 
$\tilde\epsilon\downarrow 0$ generalized It\^o's Lemma gives,
\begin{multline*}
\mathcal{L}_{\mathfrak u,\mathfrak u'}(\mathbf{Z})=
\sum_{i=1}^{I}\sum_{m=1}^M\exp(-\rho_i m)\beta_iW_i(\mathfrak u)Z_{im}(\mathfrak u,w)\\
+g^a[\mathfrak{u},\mathbf{Z}(\mathfrak{u},w)]+ 
g_{\mathfrak{u}}^a[\mathfrak{u},\mathbf{Z}(\mathfrak{u},w)] \\
+g_{\mathbf{Z}}^a[\mathfrak u,\mathbf{Z}(\mathfrak u,w)] 
\{\bm{\mu}[\mathfrak u,\mathbf{W}(\mathfrak u),\mathbf{Z}(\mathfrak u,w)] 
+\exp[\varphi(\mathfrak u)\sqrt{8/3}]\}+\\ 
\mbox{$\frac{1}{2}$}\ \sum_{i=1}^I\sum_{j=1}^I\  \hat{\bm\sigma}^{ij}[\mathfrak u,\bm{\sigma}_2^*,\mathbf{W}(\mathfrak u),\mathbf{Z}(\mathfrak u,w)]\ g_{Z_iZ_j}^a[\mathfrak u,\mathbf{Z}(\mathfrak u,w)]+o(1),
\end{multline*}
where 
$\hat{\bm\sigma}^{ij}\left[\mathfrak u,\bm{\sigma}_2^*,\mathbf{W}(\mathfrak u),\mathbf{Z}(\mathfrak u,w)\right]$ represents $\{i,j\}^{th}$ component of the variance-covariance matrix, $g_{\mathfrak u}^a=\partial g^a/\partial\mathfrak u$, 
$g_{\mathbf{Z}}^a=\partial g^a/\partial \mathbf{Z}$, 
$g_{Z_iZ_j}^a=\partial^2 g^a/(\partial Z_i\partial Z_j)$, 
$\Delta B_i\Delta B_j=\delta^{ij}\tilde\epsilon$, 
$\Delta B_i\tilde\epsilon=\tilde\epsilon\Delta B_i=0$, and 
$\Delta Z_i(\mathfrak u)\Delta Z_j(\mathfrak u)=\tilde\epsilon$, 
where $\delta^{ij}$ is the Kronecker delta function. There exists a vector $\mathbf{\xi}_{(2I\times\widehat U\times I)\times 1}$  so that 
$
\mathbf{Z}(\mathfrak u,w)_{(2I\times\widehat U\times I)\times 1}=\mathbf{Z}(\mathfrak u',w)_{(2I\times\widehat U\times I)\times 1}+\xi_{(2I\times\widehat U\times I)\times 1}. 
$
Assuming $|\xi|\leq\eta\tilde\epsilon [\mathbf{Z}^T(\mathfrak u,w)]^{-1}$ and defining a $C^2$ function
\begin{eqnarray*}
f^a[\mathfrak u,\mathbf W(\mathfrak u),\xi]
 & = & \sum_{i=1}^{I}\sum_{m=1}^M\exp(-\rho_i m) \beta_i W_i(\mathfrak u) 
 [Z_{im}(\mathfrak u',w)+\xi] \\
& & +g^a[\mathfrak u,\mathbf{Z}(\mathfrak u',w)+\xi]+ 
g_{\mathfrak u}^a[\mathfrak u,\mathbf{Z}(\mathfrak u',w)+\xi] \\
 & & +g_{\mathbf{Z}}^a[\mathfrak u,\mathbf{Z}(\mathfrak u',w)+\xi]
\{\bm{\mu}[\mathfrak u,\mathbf{W}(\mathfrak u),\mathbf{Z}(\mathfrak u',w)+\xi] \\
& & +\exp[\varphi(\mathfrak u)\sqrt{8/3} ]\} \\ 
 & & +\mbox{$\frac{1}{2}$}\sum_{i=1}^I\sum_{j=1}^I
 \hat{\bm\sigma}^{ij}[\mathfrak u,\bm{\sigma}_2^*,\mathbf{W}(\mathfrak u),\mathbf{Z}
 (\mathfrak u',w)+\xi]\times \\
 & & g_{Z_iZ_j}^a[\mathfrak u,\mathbf{Z}(\mathfrak u',w)+\xi],
\end{eqnarray*} 
we get,
\begin{eqnarray*}
\Psi_{\mathbf{u}}(\mathbf{Z})+
\tilde\epsilon\frac{\partial \Psi_{\mathfrak u}(\mathbf{Z})}{\partial \mathfrak u} & = & 
\frac{\Psi_{\mathfrak u}(\mathbf{Z})}{N_{\mathfrak u}} 
\int_{\mathbb{R}^{2I\times \widehat U\times I}} 
\exp\{-\tilde\epsilon f^a[\mathfrak u,\mathbf{W}(\mathfrak u),\xi]\}d\xi+ \\
 & &
 \hspace{-1cm}
 \frac{1}{N_{\mathfrak u}}\frac{\partial \Psi_{\mathfrak u}(\mathbf{Z})}{\partial \mathbf{Z}}
\int_{\mathbb{R}^{2I\times\hat{U}\times I}} 
\xi\exp\{-\tilde\epsilon f^a[\mathfrak{u},\mathbf{W}(\mathfrak{u}),\xi]\}d\xi+ 
o(\tilde\epsilon^{1/2}).
\end{eqnarray*}
For $\tilde\epsilon\downarrow0$, $\Delta \mathbf{Z}\downarrow0$,
\begin{eqnarray*}
f^a[\mathfrak u,\mathbf{W}(\mathfrak u),\xi] & = & 
f^a[\mathfrak u,\mathbf{W}(\mathfrak u),\mathbf{Z}(\mathfrak u',w)]+\\
& & \sum_{i=1}^{I}f_{Z_i}^a[\mathfrak u,\mathbf{W}(\mathfrak u),\mathbf{Z}(\mathfrak u,w)][\xi_i-Z_i(\mathfrak u,w)]+\\
& & 
 \hspace{-2.5cm}
\mbox{$\frac{1}{2}$}\sum_{i=1}^{I}\sum_{j=1}^{I}
f_{Z_iZ_j}^a[\mathfrak u,\mathbf{W}(\mathfrak u),\mathbf{Z}(\mathfrak u',w)]
[\xi_i-Z_i(\mathfrak u',w)][\xi_j-Z_j(\mathfrak u',w)]+o(\tilde\epsilon).
\end{eqnarray*}
There exists a symmetric, positive definite and non-singular Hessian matrix $\mathbf{\Theta}_{[2I\times\widehat U\times I]\times [2I\times\widehat U\times I]}$ and a vector $\mathbf{R}_{(2I\times\widehat U\times I)\times 1}$ such that,
\begin{multline*}
\Psi_{\mathfrak u}(\mathbf{Z})+\tilde\epsilon 
\frac{\partial \Psi_{\mathfrak u}(\mathbf{Z})}{\partial \mathfrak{u}}=
\frac{1}{N_{\mathfrak u}} 
\sqrt{\frac{(2\pi)^{2I\times\widehat U\times I}}{\tilde\epsilon |\mathbf{\Theta}|}} \times\\
\exp\{-\tilde\epsilon f^a[\mathfrak u,\mathbf{W}(\mathfrak u),\mathbf{Z}(\mathfrak u',w)]+
\mbox{$\frac{1}{2}$}\tilde\epsilon\mathbf{R}^T\mathbf{\Theta}^{-1}\mathbf{R}\}\times\\
\left\{\Psi_{\mathfrak u}(\mathbf{Z})+
[\mathbf{Z}(\mathfrak u',w)+\mbox{$\frac{1}{2}$}(\mathbf{\Theta}^{-1} \mathbf{R})]
\frac{\partial \Psi_{\mathfrak{u}}(\mathbf{Z})}{\partial \mathbf{Z}}
\right\}+o(\tilde\epsilon^{1/2}).
\end{multline*}
Assuming 
$
N_u=\sqrt{(2\pi)^{2I\times\widehat U\times I}/\left(\epsilon |\mathbf{\Theta}|\right)}>0,
$
we get Wick rotated Schr\"odinger type equation as,
\begin{multline*}
\Psi_{\mathfrak u}(\mathbf{Z})+
\tilde\epsilon\frac{\partial \Psi_{\mathfrak u}(\mathbf{Z})}{\partial\mathfrak u} 
=\{1-\tilde\epsilon f^a[\mathfrak u,\mathbf{W}(u),\mathbf{Z}(\mathfrak u',w)]+
\mbox{$\frac{1}{2}$}\tilde\epsilon\mathbf{R}^T\mathbf{\Theta}^{-1}\mathbf{R}\}\times \\
\left[\Psi_{\mathfrak u}(\mathbf{Z})+
[\mathbf{Z}(\mathfrak u',w)+\mbox{$\frac{1}{2}$}(\mathbf{\Theta}^{-1} \mathbf{R})]
\frac{\partial \Psi_{\mathfrak{u}}(\mathbf{Z})}{\partial \mathbf{Z}}
\right]+o(\tilde\epsilon^{1/2}).
\end{multline*}
As $\mathbf{Z}(\tau,w)\leq\eta\tilde\epsilon|\xi^T|^{-1}$, there exists $|\mathbf{\Theta}^{-1}\mathbf{R}|\leq 2 \eta\tilde\epsilon|1-\xi^T|^{-1}$ such that for 
$\tilde\epsilon\downarrow 0$ we have $\big|\mathbf{Z}(\mathfrak u',w)+\mbox{$\frac{1}{2}$}\left(\mathbf{\Theta}^{-1}\ \mathbf{R}\right)\big|\leq\eta\tilde\epsilon$, $|\mathbf{\Theta}^{-1}\mathbf{R}|\leq 2 \eta\tilde\epsilon|1-\xi^T|^{-1}$ and,
\[
\frac{\partial }{\partial W_i}
f^a[\mathfrak u,\mathbf{W}(\mathfrak u),\mathbf{Z}(\mathfrak u',w)]=0.
\]
We know, $\mathbf{Z}(\mathfrak u',w)=\mathbf{Z}(\mathfrak u,w)-\xi$ and for $\xi\ra 0$ as we are looking for some stable solution. Hence, $\mathbf{Z}(\mathfrak u',w)$ can be replaced by $\mathbf{Z}(\mathfrak u,w)$ and,
\begin{multline*}
\sum_{i=1}^{I}\sum_{m=1}^M\exp(-\rho_i m)\beta_iZ_{im}(\mathfrak u,w)\\
+g_{\mathbf{Z}}^a[\mathfrak u,\mathbf{Z}(\mathfrak u,w)] 
\frac{\partial
\{\bm\mu[\mathfrak u,\mathbf{W}(\mathfrak u),\mathbf{Z}(\mathfrak u,w)]+
\exp[\varphi(\mathfrak u)\sqrt{8/3}]\} 
}{\partial \mathbf{W}}
\frac{\partial \mathbf{W}}{\partial W_i }\\ 
+\mbox{$\frac{1}{2}$} 
\sum_{i=1}^I\sum_{j=1}^I 
\frac{\partial
\hat{\bm\sigma}^{ij}[u,\bm{\sigma}_2^*,\mathbf{W}(\mathfrak u),\mathbf{Z}(\mathfrak u,w)] }
{\partial \mathbf{W}}
\frac{\partial \mathbf{W}}{\partial W_i }
g_{Z_iZ_j}^a[\mathfrak u,\mathbf{Z}(\mathfrak u,w)]=0.
\end{multline*}
If $\beta_i=\beta_j=\beta^*$ for all $i\neq j$ then,
\begin{eqnarray*}
\be^*(\mathbf{Z}) & = & 
-\left[\sum_{i=1}^{I}\sum_{m=1}^M\exp(-\rho_i m)Z_{im}(\mathfrak u,w)\right]^{-1}\times 
\notag \\
 & & \left[\frac{\partial g^a[\mathfrak u,\mathbf{Z}(\mathfrak u,w)]}{\partial{\mathbf{Z}}}  
 \frac{\partial
 \{\bm{\mu}[\mathfrak u,\mathbf{W}(\mathfrak{u}),\mathbf{Z}(\mathfrak u,w)] +
 \exp[\varphi(\mathfrak u)\sqrt{8/3}]\} }{\partial \mathbf{W}}
 \frac{\partial \mathbf{W}}{\partial W_i } \right. \notag \\ 
 & & \left.
 +\mbox{$\frac{1}{2}$}
 \sum_{i=1}^I\sum_{j=1}^I
 \frac{\partial \hat{\bm\sigma}^{ij}[\mathfrak u,\bm{\sigma}_2^*,\mathbf{W}(\mathfrak u),
 \mathbf{Z}(\mathfrak u,w)]}{\partial \mathbf{W}}
\frac{\partial \mathbf{W}}{\partial W_i }
\frac{\partial^2 g^a[\mathfrak u,\mathbf{Z}(\mathfrak u,w)]}{\partial Z_i\partial Z_j} \right].
\end{eqnarray*}
\end{proof}	

\section{Discussion}
In this paper we use a Feynman path integral technique to determine the optimality coefficient $\be^*(\mathbf Z)$ for a one-day match with rain interruption and no interruptions. This coefficient tells us how to select a player to bat at a certain position based on the condition of the match. We consider different types of environmental conditions such as the speed of the wind, the moisture on the field, the speed of the ball on the outfield, extra swing from the bowler and hard time to grip a ball during a delivery for a spinner.  In the second part we focus on more  on volatile environment after the rain stops. We assume after the stoppage that the occurrence of each over strictly depends on the amount of rain at that sample over. Using It\^ o's lemma we define a $\delta_{\mathfrak u}$-gauge which generates a sample over $\mathfrak u$ instead of an actual over $u$ is assumed to follow a Wiener process. Furthermore, we assume the strategy space of a bowler from the opposition team has a $\sqrt{8/3}$-Liouville like quantum gravity surface and, we construct a stochastic It\^o-Henstock-Kurzweil-McShane-Feynman-Liouville type path integral to solve for the optimality coefficient. 
\bibliography{bib}
\end{document}